\documentclass[11pt,centertags]{amsart}

\usepackage{amsmath}
\usepackage{amsthm}
\usepackage{amssymb}
\usepackage{amstext}
\usepackage{amscd}
\usepackage{mathrsfs} 
\usepackage[pdftex,a4paper]{hyperref} 
\usepackage{comment}
\usepackage[all]{xy}
\usepackage{tikz}
\usepackage{exscale}
\usepackage{verbatim}
\usepackage{enumerate}
\usepackage{graphics}

\usepackage{fouriernc}
\usepackage[T1]{fontenc}

\numberwithin{equation}{section}

\allowdisplaybreaks

\input xy
\xyoption{all}
\theoremstyle{definition}
	\newtheorem{definition}{Definition} 
	\newtheorem*{definition*}{Definition}
	\newtheorem{example}[definition]{Example}
	\numberwithin{definition}{section}
\theoremstyle{plain}
	\newtheorem{lem}[definition]{Lemma}
	
	\newtheorem{prop}[definition]{Proposition}
	\newtheorem{thm}[definition]{Theorem}
	\newtheorem*{theorem*}{Theorem}
	\newtheorem{cor}[definition]{Corollary}

	\newtheorem*{claim*}{Claim}
	\newtheorem{problem}[definition]{Problem}

\theoremstyle{remark}

	\newtheorem{rem}[definition]{Remark}

\renewcommand{\phi}{\varphi}

\newcommand{\Z}{\mathbb{Z}}

\newcommand{\R}{\mathbb{R}}

\renewcommand\phi{\ensuremath{\varphi}}
\renewcommand\epsilon{\ensuremath{\varepsilon}}

\numberwithin{equation}{section}

\renewcommand{\emph}{\textbf}

\numberwithin{equation}{section}

\begin{document}

\title[Counting functions]{Relations between counting functions\\ on free groups and free monoids}
\author{Tobias Hartnick and Alexey Talambutsa}
\date{\today}
\begin{abstract} We consider finite sums of counting functions on the free group $F_n$ and the free monoid $M_n$ for $n \geq 2$. Two such sums are considered equivalent if they differ by a bounded function. We find the complete set of linear relations between equivalence classes of sums of counting functions and apply this result to construct an explicit basis for the vector space of such equivalence classes. Moreover, we provide a graphical algorithm to determine whether two given sums of counting functions are equivalent. In particular, this yields an algorithm to decide whether two sums of Brooks quasimorphisms on $F_n$ represent the same class in bounded cohomology.
\end{abstract}
\maketitle

\section{Introduction}
\subsection{Counting functions on free groups and free monoids}

Let $S = \{a_1, \dots, a_n\}$ be a finite set of cardinality $n\geq2$. We denote by $M_n$ the free monoid on $S$, i.e. the collection of all words over $S$ (including the empty word $e$), and by $F_n$ the free group on $S$, i.e. the collection of all reduced words over the extended alphabet $\bar{S} := \{a_1, \dots, a_n, a_1^{-1}, \dots, a_n^{-1}\}$.

Formally, a non-trivial element $v = s_1\cdots s_l \in M_n$ is called a \emph{subword} of $w = r_1\cdots r_m\in M_n$ if there exists $j \in \{1, \dots, m-l\}$ such that
\begin{equation}\label{subword}
s_i = r_{j+i} \quad \text{ for all } i = 1, \dots, l.
\end{equation}
Similarly, an element $v \in F_n$ is called a \emph{reduced subword} of an element $w \in F_n$ if the reduced word over $\bar{S}$ representing $v$ is a subword (in the above sense) of the the reduced word representing $w$. These subword relations are among the most basic relations in combinatorial (semi-)group theory.

In this article we are interested in the following quantitative refinement of the subword relation. Given $v = s_1\cdots s_l \in M_n\setminus\{e\}$ and $w = r_1\cdots r_m\in M_n$, we
denote by $\rho_v(w)$ the number of $j \in \{1, \dots, m-l\}$ such that \eqref{subword} holds. This then defines a function $\rho_v: M_n \to \mathbb N_0$ called the \emph{$v$-counting function}\footnote{Since we allow the occurrences of $v$ in $w$ to overlap, the function $\rho_v$ is sometimes called the \emph{overlapping $v$-counting function} or the \emph{big $v$-counting function}, cf. \cite{scl}.}. For example, $\rho_{a_1a_2a_1}(a_1a_2a_1a_2a_1) = 2$. Restricting to reduced words, we similarly obtain a counting function $\rho_v: F_n \to \mathbb N_0$ for every $v \in F_n\setminus\{e\}$.  It is convenient to extend the definitions to the empty word by defining $\rho_e(w)$ to be the word length $|w|_S$ of $w$ with respect to $S$.

In the sequel we write ${\mathcal C}(M_n)$ respectively ${\mathcal C}(F_n)$ for the space of real-valued functions on $M_n$ respectively $F_n$ which is spanned by the corresponding collection of counting functions $\{\rho_v\}$. Our starting point is the following simple observation which goes back (at least) to \cite{Brooks}.
\begin{prop}\label{TrivialBases} The counting functions $\{\rho_v\mid v \in M_n \setminus\{e\}\}$ form a basis for ${\mathcal C}(M_n)$, and similarly the counting functions $\{\rho_v\mid v \in F_n \setminus\{e\}\}$ form a basis for ${\mathcal C}(F_n)$.
\end{prop}
\begin{proof} On the one hand, the counting functions above span the spaces in question since
\[
\rho_e = \sum_{|w|_S = 1} \rho_w.
\]
Concerning linear independence, let $\alpha=\sum\limits_{w\in W \setminus\{e\}} \alpha_w \rho_w$ be a finite sum and let $v$ be an element of minimal length in $W$ with $\alpha_v\ne 0$. Then $\alpha(v) = \alpha_v$, hence it is possible to compute the coefficients $\alpha_v$ inductively, and the desired linear independence follows.
\end{proof}

From now on we call elements $f,g \in {\mathcal C}(M_n)$ \emph{equivalent} if they differ by a bounded function, and similarly for elements of ${\mathcal C}(F_n)$. We then denote by $\widehat{\mathcal C}(M_n)$ respectively $\widehat{\mathcal C}(F_n)$ the corresponding spaces of equivalence classes. These quotient spaces appear naturally in a number of applications, e.g. in bounded cohomology. They also admit natural interpretations as function spaces spanned by certain cyclic counting functions, see Theorem \ref{ThmCyclicFcts} in the appendix. In analogy with Proposition \ref{TrivialBases} we are going to study the following problem.

\begin{problem}\label{MainProblem} Find explicit bases for the quotient spaces $\widehat{\mathcal C}(M_n)$ and $\widehat{\mathcal C}(F_n)$.
\end{problem}

We present a complete solution to the problem in Theorem \ref{ThmBasis} below. Initially, our interest in this problem was motivated from a specific problem concerning the second bounded cohomology of free groups, which we describe in the next subsection. However, we believe that the problem is also of independent interest within the theory of combinatorics of words.

\subsection{Motivation from bounded cohomology} Historically, the need to understand the quotient space $\widehat{\mathcal C}(F_n)$ first arose from the study of the second bounded group cohomology of $F_n$ in the sense of \cite{Gromov, Ivanov}. We briefly recall this motivation here. A function $\phi: F_n \to \R$ is called a \emph{quasimorphism}\footnote{In the Russian literature such functions are sometimes called \emph{quasi-characters} (see \cite{Faiziev},\cite{Grigorchuk}), apparently following a suggestion by Shtern \cite{Shtern}. Another term in \cite{Burger-Ozawa-Thom}, referring to the stability question of Ulam \cite[Chapter 6.1]{Ulam}, is \emph{$\delta$-homomorphism}. } if
\[
\sup_{g,h \in F_n} |\phi(gh) -\phi(g) -\phi(h)| < \infty.
\]
Now if $\mathcal Q(F_n)$ denotes the space of all quasimorphisms on $F_n$ and $H^2_b(F_n; \R)$ denotes the second bounded cohomology of $F_n$ with trivial real coefficients, then there is an isomorphism (see e.g. \cite{scl})
\begin{equation}\label{BoundedCohomology}
H^2_b(F_n; \R) \cong \mathcal Q(F_n)/({\rm Hom}(F_n, \R) \oplus \ell^\infty(F_n)).
\end{equation}
In his famous paper \cite{Brooks}, R. Brooks (following earlier work of Rhemtulla \cite{Rhemtulla}) pointed out that the symmetrized counting functions
\[\phi_w := \rho_w-\rho_{w^{-1}}: F_n \to \Z\]
are quasimorphisms. It thus follows from \eqref{BoundedCohomology} that if we denote by $\widehat{\mathcal B}(F_n)$ the subspace of $\widehat{\mathcal C}(F_n)$ spanned by the equivalence classes of the \emph{Brooks quasimorphisms} $\phi_w$, then the quotient $\widehat{\mathcal B}(F_n)/{\rm Hom}(F_n, \R)$ embeds into $H^2_b(F_n; \R)$. In particular,
\[
\dim H^2_b(F_n; \R) \geq \dim \widehat{\mathcal B}(F_n) -n.
\]
Brooks claimed in \cite{Brooks} that (by an argument similar to the proof of Proposition \ref{TrivialBases}) the classes $[\phi_w] \in \widehat{\mathcal C}(F_n)$ were linearly independent except for the obvious anti-symmetry relations
\begin{equation}\label{SymmetryRelationIntroduction}
\phi_w = -\phi_{w^{-1}},
\end{equation}
and deduced that $\dim H^2_b(F_n; \R) = \infty$, thereby providing the first example of a group with infinite-dimensional second bounded cohomology. However, as Grigorchuk pointed out in \cite[p.139]{Grigorchuk}, the  linear combination \[\phi_{a_1a_2} + \phi_{a_1^{-1}a_2} + \phi_{a_1a_2^{-1}} + \phi_{a_1^{-1}a_2^{-1}}\]
is bounded in absolute value by $1$, hence yields a counterexample to the claim of Brooks. Nevertheless it is true that $\dim H^2_b(F_n; \R) = \infty$. Historically, the first complete proof was given by Mitsumatsu \cite{Mitsumatsu} who proved linear independence of an infinite collection of equivalence classes of Brooks quasimorphisms (see see e.g. \cite{Nica} for a modern treatment). Mitsumatsu's result was later extended by Faiziev \cite{Faiziev} and Grigorchuk \cite{Grigorchuk} who exhibited larger collection of linearly independent elements. Despite these efforts, the problem of finding a basis of $\widehat{\mathcal B}(F_n)$ remained open ever since. 

It turns out that a basis for $\widehat{\mathcal B}(F_n)$ can be constructed quite easily from a suitable basis of $\widehat{\mathcal C}(F_n)$. This motivated us to study the space $\widehat{\mathcal C}(F_n)$ and, by analogy, $\widehat{\mathcal C}(M_n)$.

\subsection{Relations between counting functions}
To state our results, we introduce the following notation. Denote by $\R[M_n]$ the space of finitely supported real-valued functions on $M_n$, and note that every element of $\R[M_n]$ can be written uniquely as a sum
\[\sum_{g \in M_n}\lambda_g \delta_g,\]
where $\delta_g$ is the function taking value $1$ at $g$ and $0$ elsewhere and $\lambda_g = 0$ for almost all $g \in M_n$. We thus have a canonical linear surjection
\[
\mathfrak q: \R[M_n] \to \widehat{\mathcal C}(M_n), \quad \mathfrak q\left(\sum_{g \in M_n}\lambda_g \delta_g \right) = \left[\sum_{g \in M_n}\lambda_g\rho_g\right],
\]
and we can think of the kernel $K(M_n)$ of $\mathfrak q$ as the space of relations satisfied by sums of counting functions in the quotient space $\widehat{\mathcal C}(M_n)$. By the same formula we also define a map $q: \R[F_n] \to  \widehat{\mathcal C}(F_n)$, whose kernel $K(F_n)$ describes the relations satisfied by sums of counting functions in $ \widehat{\mathcal C}(F_n)$. Finally, there is also a symmetrized version of this map, which parametrizes the Brooks space $ \widehat{\mathcal B}(M_n)$ and is given by
\[
q_{\rm sym}: \R[F_n] \to \widehat{\mathcal B}(F_n), \quad q_{\rm sym}\left(\sum_{g \in G}\lambda_g \delta_g \right) = \left[\sum_{g \in G}\lambda_g\phi_g\right].
\]
Its kernel $K_{\rm sym}(F_n)$ parametrises the relations between Brooks quasimorphisms. Our first result describes the relation spaces $K(M_n)$, $K(F_n)$ and  $K_{\rm sym}(F_n)$ explicitly.
\begin{thm}[Linear relations between counting functions]
\label{RelationsObvious}
Given $w \in M_n$, define $\{l_w, r_w\} \subset \R[M_n]$ by
\begin{equation}\label{KirchhoffMonoid}
l_w := \delta_w - \sum_{s \in S}\delta_{sw} \quad
\text{and}
\quad
r_w := \delta_w - \sum_{s \in S}\delta_{ws}.
\end{equation}
Given $w \in F_n$ with initial letter $w_1$ and final letter $w_{\rm fin}$, define
\begin{equation}\label{KirchhoffGroup}
l_w := \delta_{w} -  \negthickspace\negthickspace\negthickspace\negthickspace{\sum_{s \in \bar{S} \setminus \{ w_1^{-1}\}}} \delta_{sw},
\quad
r_w := \delta_{w} - \negthickspace\negthickspace\negthickspace\negthickspace {\sum_{s \in \bar{S} \setminus \{ w_{\rm fin}^{-1}\}}} \delta_{ws}\quad \text{and}
\quad
s_w :=\delta_w + \delta_{w^{-1}}.
\end{equation}
Then the relation spaces $K(M_n)$, $K(F_n)$ and $K_{\rm sym}(F_n)$ defined above admit the following spanning sets\footnote{When we reported this result to Danny Calegari, he kindly pointed out to us that Part (iii) of the Theorem can also be deduced from results presented in the preprint version (but not in the published version) of his joint article with Alden Walker \cite{CalegariWalker}.}:
\begin{itemize}
\item[(i)] The space $K(M_n) = \ker(\mathfrak q)$ is spanned by the set $\bigcup_{w \in M_n} \{l_w, r_w\}$.
\item[(ii)] The space $K(F_n) = \ker(q)$ is spanned by the set $\bigcup_{w \in F_n} \{l_w, r_w\}$.
\item[(iii)] The space $K_{\rm sym}(F_n) =\ker(q_{\rm sym})$ is spanned by the set $\bigcup_{w \in F_n} \{l_w, r_w, s_w\}$.
\end{itemize}
\end{thm}
Theorem \ref{RelationsObvious} will be proved in Section \ref{SecRelations} below.
\begin{rem}
\begin{enumerate}[(1)]
\item For $w = e$ the definitions of $r_w$, $s_w$ and $l_w$ have to be understood as follows: In the monoid case we define
\[
l_e := r_e := \delta_e - \sum_{s \in S}\delta_s.
\]
In the group case we define
\[
l_e := r_e := \delta_e - \sum_{s \in \bar S}\delta_s
\]
and $s_e:= 2\delta_e$.
\item Relations similar to the relations $l_w$, $r_w$ and $s_w$ appear under many different names in the literature. We prefer the terms \emph{left-extension relation}, \emph{right-extensions relations} and \emph{symmetry relations} respectively. In different contexts, the left- and right-extension relations are sometimes called (left- and right-) \emph{Kirchhoff laws} or \emph{laws of total probability}.
\item All of these relations are essentially obvious\footnote{For the convenience of the reader we establish them in Subsection \ref{SecBasicRelations} below.}. The theorem can thus be stated informally by saying that ``there are no other relations than those following from the obvious ones''.
\item The statement of (iii) contains some redundancy. Namely, since the right-extension relations follow from the left-extension relations and the symmetry relations, the space $K_{\rm sym}(F_n)$ is already spanned by the set $\bigcup_{w \in W} \{l_w, s_w\}$ (or, equivalently, $\bigcup_{w \in W} \{r_w, s_w\}$). We stated (iii) in the above redundant form to stress the analogy with (i) and (ii).
\end{enumerate}
\end{rem}

\subsection{Explicit bases}
Using the description of the space of relations provided in the last subsection we are able to provide an explicit basis for each of the spaces $\widehat{\mathcal C}(M_n)$,  $\widehat{\mathcal C}(F_n)$ and $\widehat{\mathcal B}(F_n)$. The final result is as follows:
\begin{thm}[Basis theorem]\label{ThmBasis}
\begin{itemize}
\item[(i)] Denote by $W$ the set of all words in $M_n$ which do not start or end with $a_1$ (including the empty word). Then the classes represented by the counting functions $\{\rho_w\mid w \in W\}$ form a basis for the space $\widehat{\mathcal C}(M_n)$.
\item[(ii)] Denote by $W'$ the set of all reduced words in $F_n$ which do not start with $a_1$ or $a_2a_1^{-1}$ and do not end with $a_1^{-1}$ or $a_1a_2^{-1}$ (including the empty word),  and let $W := W' \cup \{a_1^{-1}\}$. Then the classes represented by the counting functions $\{\rho_w\mid w \in W\}$ form a basis for the space $\widehat{\mathcal C}(F_n)$.
\item[(iii)] Let $W$ as in (ii) and let $W_0 := W \cup \{a_1\} \setminus\{e\}$. Let $W_+$ be a subset of $W_0$ which intersects each pair $\{w, w^{-1}\}\subset W$ in precisely one element. Then the classes represented by the counting quasimorphisms $\{\phi_w\mid w \in W\}$ form a basis for the space $\widehat{\mathcal B}(F_n)$.
\end{itemize}
\end{thm}
We will establish Theorem~\ref{ThmBasis} in Section~\ref{SecBasis} below. Parts (i) and (ii) solve Problem \ref{MainProblem}, and Part (iii) solves the long-standing problem of finding an explicit basis for the Brooks space $\widehat{\mathcal B}(F_n)$. There are of course many possible choices for $W^+$. Concretely, one can choose an order on $\bar S$ and order $F_n$ lexicographically. For any such choice, the classes represented by the counting functions associated with the words
\[
W^+ = \{w \in W_0 \mid w < w^{-1}\}
\]
form a basis.

\subsection{Algorithms for comparing counting functions} For efficient computations in $\widehat{\mathcal C}(M_n)$ and $\widehat{\mathcal C}(F_n)$ (and its subspace $\widehat{\mathcal B}(F_n)$) it is crucial to be able to decide efficiently whether two given elements of $\mathcal C(M_n)$ or $\mathcal C(F_n)$ represent the same element in $\widehat{\mathcal C}(M_n)$ or $\widehat{\mathcal C}(F_n)$. Since subtraction of counting functions can be done efficiently, this problem amounts to deciding whether a sum of the form
\begin{equation}\label{Introf}
f=\sum \alpha_w\rho_w,
\end{equation}
represents the zero class in  $\widehat{\mathcal C}(M_n)$ or $\widehat{\mathcal C}(F_n)$. In principal, Theorem \ref{ThmBasis} allows us to solve this problem, by expanding $f$ in our explicit basis. However, such an expansion, if done naively, is not efficient in any sense. We thus provide in Section \ref{SecAlgorithmsNaive} an algorithm to decide triviality of $[f]$ in a much faster way. The algorithm is based on an interpretation of sums of the form \eqref{Introf} as finite weighted trees, which we discuss in Section \ref{SecTrees}. It turns out that for many trees one can decide immediately from looking at the picture whether the corresponding sum represents a non-trivial class, see Theorem \ref{Unbalanced} below. Our algorithm provides a graphical way to replace a weighted tree by an equivalent one in such a way that after finitely many steps the resulting tree either obviously represents a non-trivial class or obviously represents the trivial class. See Section \ref{SecAlgorithmsNaive} for a precise description of the algorithm.

In praxis, the algorithm described in Section \ref{SecAlgorithmsNaive} is very fast when properly implemented. Analyzing its runtime theoretically requires some lengthy and technical arguments in complexity theory, which are beyond the scope of the present article. We refer readers interested in these purely algorithmic aspects to the sequel article \cite{Sequel}, where we also provide a detailed runtime analysis for one possible implementation of the algorithm.

\subsection{Outlook and open problems} The present article is a first major step towards efficient computation with counting functions, and in particular, towards efficient computations in the Brooks space $\widehat{\mathcal B}(F_n)$. We would like to mention that while there are many good reasons why one would want to carry out computations in $\widehat{\mathcal B}(F_n)$, the present work is motivated by some specific problems arising from work of the first author with P.~Schweitzer \cite{Hartnick-Schweitzer} concerning the ${\rm Out}(F_n)$-action on bounded cohomology of free groups. Namely, the automorphism group of $F_n$ acts naturally on the space $\mathcal Q(F_n)/\ell^\infty(F_n)$, and this action factors through ${\rm Out}(F_n)$. There is a natural ${\rm Out}(F_n)$-invariant locally-convex (non-complete) topology on $\mathcal Q(F_n)/\ell^\infty(F_n)$ given by pointwise convergence of homogeneous representatives (cf. \cite{Grigorchuk, Hartnick-Schweitzer}). The following equivariant version of a classical result of Grigorchuk \cite{Grigorchuk} was established in \cite[Section 2]{Hartnick-Schweitzer}.
\begin{thm}[Grigorchuk, Hartnick -- Schweitzer]
The Brooks space $\widehat{\mathcal B}(F_n)$ is a dense subspace of $\mathcal Q(F_n)/\ell^\infty(F_n)$ and invariant under the action of ${\rm Out}(F_n)$. In particular, $\widehat{\mathcal B}(F_n)$ is independent of the free generating set used to define it, and the action of ${\rm Out}(F_n)$ on $\mathcal Q(F_n)/\ell^\infty(F_n)$ (and thus also the ${\rm Out}(F_n)$-action on $H^2_b(F_n; \R)$) is uniquely determined by its restriction to $\widehat{\mathcal B}(F_n)$.
\end{thm}
This motivates a closer study of the action of ${\rm Out}(F_n)$ on $\widehat{\mathcal B}(F_n)$. For example, one would like to compute the stabilizer of a general element $[f] \in \widehat{\mathcal B}(F_n)$ under ${\rm Out}(F_n)$. The ${\rm Out}(F_n)$-action on Brooks quasimorphisms is given by very explicit formulas (cf. \cite{Hartnick-Schweitzer}). However, in order to decide whether $g \in {\rm Out}(F_n)$ stabilizes $[f] \in \widehat{\mathcal B}(F_n)$, one has to be able to decide whether $g.f -f$ is bounded. By means of the algorithm developed in this article, it is now possible to decide this efficiently. We thus think that the present work provides a major step towards an understanding of the ${\rm Out}(F_n)$-action on quasimorphisms.

Of course, this is just the tip of a much larger iceberg. Analogues of Brooks quasimorphisms have been define for Gromov-hyperbolic groups \cite{EpsteinFujiwara}, various classes of groups acting on hyperbolic spaces \cite{Fujiwara, Hamenstaedt}, mapping class groups \cite{BestvinaFujiwara} and most recently for general acylindrically hyperbolic groups \cite{HullOsin, BBF}, comprising all previous constructions. In all these situations it is known that there is an infinite-dimensional subspace of the second bounded cohomology which is analogous to the Brooks space. The combinatorial fine-structure of these generalized Brooks spaces is not at all understood at this point. Even for relatively simple examples such as surface groups, we have currently no idea how a basis for the generalized Brooks space should look like.

\subsection{Organization of the article} This article is organized as follows: In Section~\ref{SecRelations} we establish Theorem \ref{RelationsObvious}, and in Section~\ref{SecBasis} we establish Theorem \ref{ThmBasis}. In both cases, we first consider the monoid case, and then deal with the additional complications in the group case. In Section~\ref{SecTrees} we explain how sums of counting functions can be represented graphically as finite weighted trees. Here the main result is Theorem \ref{Unbalanced} which singles out a large class of such trees which represent non-trivial elements in $\widehat{\mathcal C}(M_n)$ and $\widehat{\mathcal C}(F_n)$. Based on this result, we present in Section \ref{SecAlgorithmsNaive} an algorithm to decide whether a given sum of counting functions is bounded. The appendix collects some basic facts about homogenizations of counting functions used throughout the body of the text.

\subsection{Acknowledgements.} We would like to thank Tatiana Smirnova-Nagnibeda for pointing out the reference \cite{Martin} to us. We also thank Danny Calegary, Anton Hase, Pascal Schweitzer and Alden Walker for useful discussions. We thank the Technion for providing excellent working conditions during several visits of the second author. Tobias Hartnick was supported as a Taub fellow by Taub foundation and Alexey Talambutsa was supported by Swiss National Science Foundation project PP00P2-128309/1 and Russian Foundation for Basic Research.

\section{Relations between counting functions}\label{SecRelations}

\subsection{Basic relations between counting functions}\label{SecBasicRelations}  The goal of this section is to establish Theorem~\ref{RelationsObvious}, i.e. to determine all relations between counting functions and counting quasimorphisms. Parts (i), (ii) and (iii) of the theorem will be established in Subsections \ref{SubsecRelMonoids}, \ref{SubsecRelGroups} and \ref{SubsecRelQM} respectively. Informally, the theorem states that every relation between counting functions is a consequence of certain basic relations. In this subsection we briefly explain these basic relations. We start with the case of monoids.
\begin{lem} For every $w \in M_n$ the left-extension relation $l_w$ and the right-extension relation $r_w$ as defined in \eqref{KirchhoffMonoid} are contained in the relation space $K(M_n)$.
\end{lem}
\begin{proof} For $w = e \in M_n$ we have $\mathfrak q(l_w) = \mathfrak q(r_{w})=0$, since adding the counts of letters of a words yields the word length.
If $w \in M_n$ with $|w|_S \geq 1$ then for every $v \in M_n$ the difference
\[ \rho_{w}(v) - \sum_{s \in S} \rho_{sw}(v)\]
takes value $1$ or $0$ depending on whether $v$ starts with $w$ or not. Similarly,
\[ \rho_{w}(v) - \sum_{s \in S} \rho_{ws}(v)\]
takes value $1$ or $0$ depending on whether $v$ ends with $w$ or not. Since these are bounded functions we deduce that $\mathfrak q(l_{w}) = \mathfrak q(r_{w}) =0$.
\end{proof}
In the group case we have the following similar result:
\begin{lem} For every $w \in F_n$ the left-extension relation $l_w$ and the right-extension relation $r_w$ as defined in \eqref{KirchhoffGroup} are contained in the relation space $K(F_n)$. Moreover, the relations $l_w$, $r_w$ and the relation $s_w$ defined in \eqref{KirchhoffGroup} are contained in $K_{\rm sym}(F_n)$.
\end{lem}
\begin{proof} The first statement is proved exactly as in the monoid case: We have $q(l_e) =q(r_{e})=0$ since the word length can be obtained by adding up the counts for all possible letter, and if $|w|_S \geq 1$, then $q(l_w)$ and $q(r_w)$ can be represented by a function taking only values $0$ and $1$.

Concerning $K_{\rm sym}(F_n)$ we can argue as follows: Since $\phi_w = \rho_w - \rho_{w^{-1}}$ we have
\[
q_{\rm sym}(l_w) = q(l_w)-q(r_{w^{-1}})=0
\]
by the result about $K(F_n)$. Dually we obtain $q_{\rm sym}(l_w) = 0$. Finally, $q_{\rm sym}(s_w)$ is represented by the function
\[
\phi_w + \phi_{w^{-1}} =  \rho_w - \rho_{w^{-1}} +\rho_{w^{-1}}-\rho_{w}=0. \qedhere
\]
\end{proof}
From now on we denote by $B(M_n) \subset K(M_n)$ and $B(F_n) \subset K(F_n)$ the respective subspaces spanned by corresponding left- and right-extension relations $\{l_w, r_w\}$. We also denote by $B_{\rm sym}(M_n) \subset K_{\rm sym}(F_n)$ the subspace generated by the relations $\{l_w, r_w, s_w\}$. In this notation our goal is to establish the equalities $B(M_n) = K(M_n)$, $B(F_n) = K(F_n)$ and $B_{\rm sym}(M_n)=K_{\rm sym}(M_n)$.

\subsection{Pure elements}\label{SecRelationsIntro}
From now on we fix an integer $n \geq 2$ and a set $S =\{a_1, \dots, a_n\}$ of cardinality $n$. We then denote by $M_n$, respectively $F_n$, the free monoid, respectively free group, on $S$. Given an integer $L \geq 0$, an element $f \in \R[M_n]$ will be called \emph{pure} of \emph{length} $L$ if $|w| = L$ for all $w \in {\rm supp}(f)$. Thus for example $\delta_{a_1a_2}-3\delta_{a_2a_1}$ is pure of length $2$. Similarly, $\R[F_n]_L$ denotes the space of all pure finitely supported real-valued functions on $F_n$ of length $L$. We also introduce the notations
\begin{eqnarray*}
K_L(M_n) := K(M_n) \cap \R[M_n]_L, && B_L(M_n) := B(M_n) \cap K_L(M_n), \\ K_L(F_n) := K(F_n) \cap \R[F_n]_L, && B_L(F_n) := B(F_n) \cap K_L(F_n),
\end{eqnarray*}
where the basic relation spaces $B(M_n)$ and $B(F_n)$ are defined as in the previous subsection. When the monoid or group in question is clear from the context we simply write $K$, $B$, $K_L$, $B_L$. Note that the spaces $K_L$ and $B_L$ are finite-dimensional for each $L \geq 0$.
\begin{lem}\label{BEqualsK}
If $\dim B_L \geq \dim K_L$ for all $L \geq 0$, then $B=K$.
\end{lem}
\begin{proof}  Since $ B_L \subset  K_L$ the assumption of (i) implies $ B_L =  K_L$ for all $L$. Now if $r \in  K$ is any function, then by adding elements of $ B$ we can always achieve that $r$ is pure. Thus if $ B_L =  K_L$ for all $L \geq 0$, then
\[ K/ B \subseteq  \left(\sum_L  K_L\right)/ B=\sum_L  K_L/( B \cap  K_L) =\sum_L  K_L/ B_L= 0. \qedhere\]
\end{proof}
This reduces the proof of the first two parts of Theorem~\ref{RelationsObvious} to an estimate of the dimensions of the finite-dimensional vector spaces $B_L$ and $K_L$. We now carry out the necessary estimates, first  in Subsection~\ref{SubsecRelMonoids} for $M_n$ and then  in Subsection~\ref{SubsecRelGroups} for $F_n$. The argument is basically the same in both cases, but in the case of $F_n$ some additional care has to be taken because of potential cancellations. Once the relations between counting functions are determined, it is easy to also determine the relations between counting quasimorphisms. This will be carried out in Subsection~\ref{SubsecRelQM}.

\subsection{Relations between counting functions on free monoids}\label{SubsecRelMonoids}
The goal of this subsection is to establish Part (i) of Theorem~\ref{RelationsObvious} concerning the space of relations between counting functions on monoids. We fix $n \geq 2$ and consider the free monoid $M_n$ with generating set $S = \{a_1, \dots, a_n\}$. We also use the notation introduced in Subsection \ref{SecRelationsIntro} and write $K$, $B$, $K_L$ and $B_L$ for $K(M_n)$, $B(M_n)$, $K_L(M_n)$ and $B_L(M_n)$. By Lemma~\ref{BEqualsK} it suffices to establish $\dim B_L \geq \dim K_L$. For $L = 0$ we have $\dim K_L = 0$, so there is nothing to show. For $L \geq 1$ we are going to show that
\begin{equation}\label{MonoidBLvsKL}
\dim B_L \geq n^{L-1}-1 \geq \dim K_L
\end{equation}
by first establishing a lower bound for $\dim B_L$, and then establishing an upper bound for $\dim K_L$.

\textsc{Step 1 (Lower bound for $\dim B_L$):} We now establish the first inequality in \eqref{MonoidBLvsKL} for all $L \geq 1$. For $L = 1$ there is nothing to show, thus we will assume $L \geq 2$. Given any word $w$ of length $L-1$ we have
\[b_w := r_{w}-l_{w}\in B_L.\]
This defines $n^{L-1}$ elements in $B_L$, and we claim that their span $B_L^0$ has dimension precisely $n^{L-1}-1$. If we write each of the elements $b_w$ as
\[b_w = \sum_{|v| = L} \lambda_{w,v} \delta_v,\]
then this amounts to showing that the $n^{L-1} \times n^{L}$-matrix
\begin{equation}\label{ALMn}
A_{L}(M_n) = (\lambda_{w,v})\end{equation}
has rank $n^{L-1}-1$.
\begin{example} The matrix $A_3(M_2)$ has the form
\[
\begin{array}{crrrrrrrr}
&a_1a_1a_1&a_1a_1a_2&a_1a_2a_1&a_1a_2a_2&a_2a_1a_1&a_2a_1a_2&a_2a_2a_1&a_2a_2a_2\\
a_1a_1&0 & -1 & 0 & 0&1&0&0&0\\
a_1a_2& 0&1&-1&-1&0&1&0&0\\
a_2a_1& 0&0&1&0&-1&-1&1&0\\
a_2a_2&0&0&0&1&0&0&-1&0
\end{array}
\]
\end{example}
As is apparent in the example, the structure of the matrix $A_L$ is very special: observe that $\lambda_{w,v} \neq 0$ only if $w$ is a maximal proper subword of $v$. Thus each column contains either one $+1$ and one $-1$ (if deleting the first and the last letter lead to different words), or no non-zero entry at all (if deleting the first and last letter lead to the same word). The latter actually happens only if $v=w$ is a power of some $a_j$. As a first consequence of this special structure, we see that the sum of the rows is $0$, whence ${\rm rank}(A_L(M_n)) \leq n^{L-1}-1$. The converse inequality can be reformulated in graph theoretic terms:
\begin{lem} \label{GraphReformulation}
Let $\Gamma_L(M_n)$ be the graph whose vertices are words $w$ of length $L-1$, and in which two different words $w$ and $w'$ are joined by an edge iff there exists a column of the matrix $A_L(M_n)$ with non-zero entries in the rows corresponding to both $w$ and $w'$. Then ${\rm rank}(A_L(M_n))= n^{L-1}-1$ if and only if the graph $\Gamma_L(M_n)$ is connected.
\end{lem}
\begin{proof} Assume that the graph is connected and that some linear combination of rows involving the $w$-th row is $0$. Then every non-zero entry of the $w$-th row has to be cancelled. However, since for each of these entries there is only one other row containing it, this row has to be involved in the linear combination. The upshot is that if $v$ and $w$ are connected by an edge in $\Gamma_L$, then every linear combination of rows involving $w$ which adds up to $0$-row must also involve $v$, and the coefficients for $v$ and $w$ in this sum have to be the same. We deduce that if $\Gamma_L$ is connected then no proper subset of rows is linearly dependent, whence ${\rm rank}(A_L(M_n))= n^{L-1}-1$. Conversely, if $\Gamma'$ is a connected component of $\Gamma_L(M_n)$, then adding up the rows corresponding to the vertices of $\Gamma'$ yields $0$. In particular ${\rm rank}(A_L(M_n))= n^{L-1}-1$ implies that there is only one connected component.
\end{proof}
The following picture shows the graph $\Gamma_3(M_2)$ corresponding to the matrix $A_3(M_2)$ above:
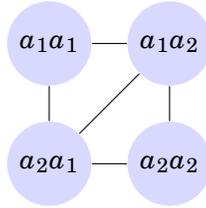
\begin{center}
\begin{tikzpicture}
  [scale=.8,auto=left,every node/.style={circle,fill=blue!15}]
  \node(0) at (3,1) {$a_2a_2$};
  \node (1) at (3,3) {$a_1a_2$};
  \node (2) at (1,1)  {$a_2a_1$};
  \node (3) at (1,3)  {$a_1a_1$};
  \foreach \from/\to in {0/1, 0/2, 1/2, 1/3, 2/3}
    \draw (\from) -- (\to);
\end{tikzpicture}
\end{center}
The graphs $\Gamma_L(M_n)$ are closely related to a family of classical examples in finite graph theory called \emph{de Bruijn graphs} \cite{deBruijn}. Recall that the \emph{$L$-th de Bruijn graph over $S$} is the graph $\Gamma_L(S)$ whose vertices are words of length $L-1$ and whose edges are words of length $L$ connecting the subwords obtained by deleting the first, respectively last letter. We claim that the graph  $\Gamma_L(M_n)$ can be obtained from $\Gamma_L(S)$ by erasing all loops and multiple edges. Indeed, vertices $w$ and $w'$ of $\Gamma_L(M_n)$ are connected by an edge if and only if there exist $s_1, s_2 \in S$ such that either $ws_1 = s_2w'$ or $s_1w = w's_2$, but not both. Then the claim follows from the fact that the latter case can only happen if if $w = w'$ is a power of some letter $a_j$. We thus refer to $\Gamma_L(M_n)$ as a \emph{loop-erased de Bruijn graph}.

Since erasing loops and multiple-edges does not change connectivity of a graph, it remains only to show connectedness of the de Bruijn graphs. This is a classical fact from finite graph theory. Explicitly,  two words $w = s_1 \cdots s_{L-1}$ and $w' = r_{1} \cdots r_{L-1}$ can be connected through the path
\[s_1\cdots s_{L-1} \sim r_{L-1}s_1\cdots s_{L-2} \sim r_{L-2}r_{L-1}s_1 \cdots s_{L-3} \sim \dots \sim r_1\cdots r_{L-1}.\]
We thus deduce from Lemma \ref{GraphReformulation} that
\[\dim B_L \geq \dim B_L^0 = {\rm rank}(A_L(M_n))= n^{L-1}-1.\]
This finishes \textsc{Step 1}.\\

\textsc{Step 2 (Upper bound for $\dim K_L$):} We are now going to establish the second inequality in \eqref{MonoidBLvsKL} for all $L \geq 1$. Rather than showing directly that $\dim K_L \leq n^{L-1}-1$ we will show that the codimension of $K_L$ in $\R[M_n]_L$ is bounded below by
\begin{equation}\label{codimMonoid}
{\rm codim}\, K_L \geq \dim \R[M_n]_L - (n^{L-1}-1) = (n-1)\cdot n^{L-1}+1.
\end{equation}
In order to establish \eqref{codimMonoid} we will construct $(n-1)\cdot n^{L-1}+1$ linearly independent linear functionals on $\R[M_n]_L$, which vanish on $K_L$. Such a linear functional will be called a \emph{certificate}.

We now describe a way to construct certificates using homogenization. Recall from the appendix that a function $f: M_n \to \R$ is called \emph{homogenizable} if the limit
\begin{equation}\label{homexplicit}
\widehat{f}(x) := \lim_{n \to \infty} \frac{f(x^n)}{n}
\end{equation}
exists for every $x \in M_n$, and in this case the function $\widehat{f}: M_n \to \R$ defined by \eqref{homexplicit} is called the \emph{homogenization} of $f$. By Corollary \ref{HomogenizationConvenient} every function $f \in \mathcal C(M_n)$ is homogenizable. Given $c \in M_n$ we me thus define a linear functional
\begin{equation}\label{cCert}
\langle c \rangle_L : \R[M_n]_L \to \R, \quad \sum \lambda_g \delta_g \mapsto \sum \lambda_g \widehat{\rho_g}(c)
\end{equation}
by evaluation of the homogenization at $c$.
\begin{lem} For every $c \in M_n$ and $L \geq 1$ the functional  $\langle c \rangle_L \in (\R[M_n]_L)^*$ is a certificate, i.e. it vanishes on $K_L$.
\end{lem}
\begin{proof} Let $f_0=\sum \lambda_g \delta_g \in K_L$. By definition, this means that the function
\[
f:=  \sum \lambda_g\rho_g \in {\mathcal C}(M_n)
\]
is bounded. Consequently, the homogenization $\widehat{f}$ satisfies $\widehat{f} \equiv 0$. We deduce that, for every $c \in M_n$,
\[
0 = \widehat{f}(c) =  \sum \lambda_g\widehat{\rho_g}(c)= \langle c \rangle_L \left( \sum_{g \in M_n} \lambda_g\delta_g\right) =  \langle c \rangle_L(f_0),
\]
which shows that $f_0 \in \ker( \langle c \rangle_L)$ and finishes the proof.
\end{proof}
In view of the lemma we refer to $\langle c\rangle_L$ as the \emph{$L$-certificate of $c$}. It remains to show that there exists $(n-1)\cdot n^{L-1}+1$ elements of $M_n$ whose corresponding $L$-certificates are linearly independent.
For this we start from the set
\[W_L := \{a_iw\mid i \neq 1, \; |w| = L-1\} \cup \{a_1^L\}\]
of \emph{special words} and define the associated set of certificates to be
\[C_L := \{ \langle ca_1^L \rangle_L \,|\, c \in W_L\}.\]
Note that $|C_L| = |W_L| = (n-1)\cdot n^{L-1} +1$. We will show that the certificates in $C_L$ are linearly independent, thereby finishing the proof. For our computations in the dual space $\R[M_n]_L^*$ we will denote by $\{[w]\mid w \in S^L\}$ the dual basis to the basis $\{\delta_w\mid w \in S^L\}$, i.e.
\[
[w](\delta_v) = \left\{\begin{array}{ll}1, & v=w\\0, &\text{else}\end{array}\right..
\]
As an immediate consequence of Lemma \ref{HomogenizationMonoid}, we can write our certificates in terms of this dual basis as follows:
\begin{lem} \label{CertificateFormula}
Let $c \in M_n$. Then
\[
\langle c \rangle_L = \sum [w],
\]
where $w$ runs through all cyclic subwords of $c$ of length $L$ with multiplicity.\qed
\end{lem}
For example,
\[
\langle a_1a_2a_1^2a_2^2 \rangle_3 = [a_1a_2a_1] + [a_2a_1a_1] + [a_1a_1a_2] + [a_1a_2a_2] +[a_2a_2a_1] + [a_2a_1a_2].
\]
Now we can finish the proof by the following lemma.
\begin{lem} The set $C_L \subset (\R[M_n]_L)^*$ is linearly independent.
\end{lem}
\begin{proof} Using again our set $W_L$ of special words we introduce a test space $T_L \subset \R[M_n]_L$ as
\[T_L := {\rm span}\{\delta_w \,|\, w \in W_L\}.\]
We will show that already the restrictions of the certificates in $C_L$ to $T_L$ are linearly independent.  For this we observe that if $w=a_is_1\cdots s_{L-1} \in W_L \setminus \{a_1^L\}$, then $i \neq 1$ and by Lemma~\ref{CertificateFormula}
\begin{eqnarray*}
\langle wa_1^L \rangle_L &=& [a_is_1 \cdots s_{L-1}] + [s_1 \cdots s_{L-1}a_1] + \cdots +  [s_{L-1}a_1^{L-1}] + [a_1^L]\\
&&+ [a_1^{L-1}a_i] + [a_1^{L-2}a_is_1] + \dots + [a_1a_is_1 \cdots s_{L-2}],
\end{eqnarray*}
Now, by definition, the words $a_1^{L-1}a_i$, $a_1^{L-2}a_is_1$, \dots, $a_1a_is_1 \cdots s_{L-2}$ appearing in the second row are not contained in $W_L$. It follows that
\begin{eqnarray*}
\langle wa_1^L \rangle_L|_{T_L} &=& ([a_is_1 \cdots s_{L-1}] + [s_1 \cdots s_{L-1}a_1] + \dots +  [s_{L-1}a_1^{L-1}] + [a_1^L])|_{T_L}.
\end{eqnarray*}
Concerning the final certificate we have
\[\langle a_1^{2L} \rangle_L=\langle a_1^{2L} \rangle_L|_{T_L} = 2L \cdot [a_1^L].\]
Now we introduce a total order on $M_2$ as follows: We first order $S $ by declaring that
\[
a_1 < a_2 < \dots < a_n
\]
and then extend to $M_n$ in a right-lexicographic (or Hebrew) way: Given $w, w' \in M_2$ and $s, s' \in S$ we set $ws < w's'$ if and only $s < s'$ or $s = s'$ and $w<w'$. Note in particular that $w_1 > w_2$ if and only if $w_1a_1^L > w_2a_1^L$. With this order understood the largest element in the support of $\langle wa_1^L \rangle_L|_{T_L} $ is precisely $w$; indeed this follows from
\[
 a_is_1 \cdots s_{L-1} \geq s_1 \cdots s_{L-1}a_1 \geq \dots \geq s_{L-1}a_1^{L-1} \geq a_1^L
\]
and the above computation of supports. We deduce that the matrix obtained by evaluating the certificates from $C_L$ on the basis $\{\delta_w \,|\, w \in W_L\}$ of $T_L$ is of lower triangular form for the given order with non-trivial diagonal entries. It therefore has full rank, and the lemma follows.
\end{proof}
This finishes \textsc{Step 2} and thereby the proof of Theorem~\ref{RelationsObvious}.(i).
\begin{rem}\label{RemarkB0} Note that as a by-product of the proof we also see that $B_L = B^0_L$.
\end{rem}
\subsection{Relations between counting functions on free groups}\label{SubsecRelGroups}
We are now going to extend the results of the previous subsection to the group case, thereby establishing Part (ii) of Theorem~\ref{RelationsObvious}. The proof is in close analogy with the monoid case and we will only highlight the necessary modifications. Throughout we fix $n \geq 2$ and write $K$, $B$, $K_L$ and $B_L$ as short hands for the spaces $K(F_n)$, $B(F_n)$, $K_L(F_n)$ and $B_L(F_n)$ as introduced in Section \ref{SecRelationsIntro}. By Lemma~\ref{BEqualsK} it suffices again to establish $\dim B_L \geq \dim K_L$.

Assume first that $L \leq 1$. We claim that in this case $\dim K_L = 0$, whence the desired inequality hold automatically. For $L = 0$ the claim follows from the fact that the function $\rho_e$ is unbounded. Now let $L = 1$. We have to show that the functions $\{\rho_{a_1}, \rho_{a_1^{-1}}, \dots, \rho_{a_n}, \rho_{a_n^{-1}}\}$ are linearly independent modulo bounded functions. For this it suffices to observe that
\[
\left(\sum_{i=1}^n \lambda_i \rho_{a_i} + \sum_{i=1}^n \mu_i \rho_{a_i^{-1}}
\right)(a_j^n) = \left\{\begin{array}{rl} n \cdot \lambda_j, & n>0\\ -n \cdot \mu_j, & n<0\end{array}\right..
\]
We have thus established the desired inequality for $L \leq 1$. Next we are going to show that
\begin{equation}\label{BoundBLKLGroup}
\dim B_L \geq  2n(2n-1)^{L-2}-1 \geq \dim K_L.
\end{equation}
for all $L \geq 2$.\\

\textsc{Step 1:} Concerning the lower bound on $\dim B_L$ we observe that, as in the case of monoids, every reduced word $w$ of length $L-1$ gives rise to a basic relation $b_w = r_{w}-l_{w}$ of length $L$, and there are $2n(2n-1)^{L-2}$ such words. Again we write each of the elements $b_w$ as
\[b_w = \sum_{|v| = L} \lambda_{w,v} \delta_v,\]
and obtain a matrix $A_{L}(F_n) = (\lambda_{w,v})$ of size $2n(2n-1)^{L-2} \times 2n(2n-1)^{L-1}$. We then have to show that the matrix $A_{L}(F_n)$ has rank $2n(2n-1)^{L-2}-1$. This amounts again to showing connectedness of a certain graph.

More precisely, let $\bar{S} := \{a_1, a_1^{-1}, \dots, a_n, a_n^{-1}\}$ be the symmetrization of the free generating set of $F_n$. Then the \emph{$L$th Martin -- de Bruijn  graph\footnote{These graphs were popularized through the PhD of Martin \cite{Martin}, who pointed out that they are Eulerian and that this can be used to show that integral measured currents on $F_n$ can be written as sums of counting currents.} over $\bar S$} is the graph $\Gamma_L(\bar S)$ with vertices given by reduced words of length $L-1$ over $\bar{S}$ and edges given by reduced words of length $L$ over $\bar{S}$, where the edge labelled by some word $w$ connects the two vertices labelled by the words which are obtained by cancelling the first, respectively last, letter of $w$. Then Lemma \ref{GraphReformulation} and its proof carry over to the present setting in the following form:
\begin{lem} \label{GraphReformulationGroup}
Let $\Gamma_L(F_n)$ be the loop-erased version of the Martin -- de Bruijn  graph $\Gamma_L(\bar S)$. Then ${\rm rank}(A_L(F_n))= n^{L-1}-1$ if and only if the graph $\Gamma_L(F_n)$ (or, equivalently, $\Gamma_L(\bar S)$) is connected.\qed
\end{lem}
Connectedness of the Martin -- de Bruijn  graph $\Gamma_L(\bar S)$ is again well-known and easy to see as follows: If $w =  s_1\cdots s_{L-1}$ and $w' = r_1 \cdots r_{L-1}$ with $r_{L-1} \neq s_1^{-1}$, then as in the monoid case,
\[
s_1\cdots s_{L-1} \sim r_{L-1}s_1\cdots s_{L-2} \sim r_{L-2}r_{L-1}s_1 \cdots s_{L-3} \sim \dots \sim r_1\cdots r_{L-1}.
\]
Thus in this case, $w$ and $w'$ are in the same connected component of the graph. If, however, $r_{L-1} = s_1^{-1}$, then we can choose $a_j$ with $a_j^{\pm 1} \neq r_{L-1}$, and by the previous case, both $w$ and $w'$ are in the same connected component as $a_j^{L-1}$.

This shows connectedness of the de Bruijn-Martin graph and thereby finishes \textsc{Step 1}.\\

\textsc{Step 2:} In analogy with the monoid case we have to construct $ \dim \R[F_n]_L-2n(2n-1)^{L-2}+1$ linearly independent certificates for $K_L$. By Lemma \ref{HomogenizationGroup} the counting functions $\rho_w \in {\mathcal C}(F_n)$ are homogenizable. We can thus define, as in the monoid case, for every $c \in F_n$ a certificate $\langle c\rangle_L$ for $K_L$ by the same formula as in \eqref{cCert}. If we assume in addition that $c$ is cyclically reduced, then we have the following analogue of Lemma~\ref{CertificateFormula}, which is again a direct consequence of Lemma \ref{HomogenizationGroup}.
\begin{lem}
Let $c \in F_n$. If $c$ is cyclically reduced, then
\[
\langle c \rangle_L = \sum [w],
\]
where $w$ runs through all cyclic subwords of $c$ of length $L$ with multiplicity.\qed
\end{lem}
In order to establish the second inequality in \eqref{BoundBLKLGroup} along the same lines as in the monoid case we will thus have to find $ \dim \R[F_n]_L-2n(2n-1)^{L-2}+1$ \emph{cyclically reduced} words with linearly independent $L$-certificates. For this we choose our set of special words as follows: Denote by $\bar{S}^{(L)} \subset \bar{S}^L$ the set of all reduced words of length $L$, by $A^{(1)}_L$ the set of all such words starting with $a_1$ and by $A^{(2)}_L$ the set of all such words starting with $a_2 a_1^{-1}$. Then we define
\[
W_L := \bar{S}^{(L)}\setminus(A^{(1)}_L\cup A^{(2)}_L) \cup\{a_1^L\}.
\]
Note that
\begin{eqnarray*}
|W_L| &=& |\bar{S}^{(L)}|-|A^{(1)}_L| - |A^{(2)}_L| + 1\\ &=&  \dim \R[F_n]_L-(2n-1)^{L-1}-(2n-1)^{L-2}+1\\ &=&
 \dim \R[F_n]_L-2n(2n-1)^{L-2}+1.
\end{eqnarray*}
We are now going to define a certificate for every $w \in W_L$. Since we want to avoid cyclic cancellation, the definition of the set of certificates $C_L$ is more complicated than in the monoid case. Given $w \in W_L$ we define a reduced word $s(w) \in  \bar{S}^{(L)}$ as follows:
\begin{itemize}
\item[(i)] If $w$ does not start with $a_1^{-1}$ and does not end with $a_1^{-1}$, then $s(w):=a_1^L$.
\item[(ii)] If $w$ does not start with $a_1^{-1}$ but ends with $a_1^{-1}$, then $s(w) := a_2 a_1^L$.
\item[(iii)] If $w$ starts with $a_1^{-1}$ but does not end with $a_1^{-1}$, then $s(w) :=  a_1^L a_2$.
\item[(iv)] If $w$ starts and ends with $a_1^{-1}$, then $s(w) := a_2 a_1^L a_2$.
\end{itemize}
This definition is made in such a way that for every $w \in W_L$ the word $ws(w)$ is cyclically reduced, and we define
\[
C_L := \{\langle ws(w) \rangle_L \mid w \in W_L\}.
\]
It remains to show only that the set $C_L$ is linearly independent. We will in fact show that the certificates in $C_L$ are linearly independent when restricted to the test space $T_L := {\rm span}\{\delta_w\,|\,w \in W_L\}$. We order the words in $W_L$ as follows: first we order the letters according to
\[
 a_1 < a_2 < ... < a_n < a_1^{-1} < a_2^{-1} < \ldots < a_n^{-1}
\]
and extend this to a total ordering on $W_L$ in the right-lexicographic (or Hebrew) way. We now claim that the matrix $N_L$ formed by the evaluations of the certificates in $C_L$ on the canonical basis of the test space $T_L$
is a lower triangular matrix with respect to the given order on $W_L$ with non-zero diagonal entries. Indeed, this follows from the following two basic observations:
\begin{itemize}
\item[(i)] Consider the cyclic $L$-subwords of $ws(w)$ with initial letter in $w$. These form a strictly decreasing sequence, so $w$ is the biggest of them.
\item[(ii)] It remains to deal with those cyclic $L$-subwords $v$ of $ws(w)$ whose initial letter is contained in $s(w)$. Here there are several cases: If $w$ is among the words of Type (i) or (iii), then $v$ is not in $W_L$, so we can ignore it. In cases (ii) and (iv), $v$ is either not in $W_L$ or $v=a_2 a_1^{L-1}$. Since $w$ is not of type (i) we have $w \neq a_1^L$. But $v$ is the second smallest element of $W_L$ after $a_1^L$, so $w \geq v$ also in this case.
\end{itemize}
This shows that the matrix is indeed lower triangular with non-zero coefficients on the diagonal, which finishes \textsc{Step 2} and thereby the proof of Theorem~\ref{RelationsObvious}.(ii).
\subsection{Relations between counting quasimorphisms}\label{SubsecRelQM}
In this subsection we finish the proof of Theorem~\ref{RelationsObvious} by deducing Part (iii) of the theorem from Part (ii). Throughout we fix $n \geq 2$ and, using the notation introduced in Subsection \ref{SecRelationsIntro}, write $B$, $K$, $\widehat{\mathcal C}$, $K_{\rm sym}$ and $\widehat{\mathcal B}$ for
$B(F_n)$, $K(F_n)$, $\widehat{\mathcal C}(F_n)$, $K_{\rm sym}(F_n)$ and $\widehat{\mathcal B}(F_n)$ respectively. We also denote by $B_{\rm sym} \subset K_{\rm sym}$ the space spanned by he symmetry relations $\{s_w\mid w \in F_n\}$. We have already seen in Subsection \ref{SecBasicRelations} that $B+B_{\rm sym} \subset K_{\rm sym}$, and we would like to show the opposite inclusion $K_{\rm sym} \subset B+B_{\rm sym}$.

For this we first define a linear involution \[
\sigma: \R[F_n] \to \R[F_n], \quad f\mapsto f^*
\]
by demanding that $\delta_w^* := -\delta_{w^{-1}}$. Then the natural inclusion map $\iota:\widehat{\mathcal B} \to \widehat{\mathcal C}$, which sends the class $[\phi_w]$ to the class  $[\rho_w-\rho_{w^{-1}}]$ lifts to a map
\[
i: \R[F_n] \to \R[F_n], \quad f \mapsto f+f^*,
\]
i.e. we get a commutative diagram
\[\begin{xy}\xymatrix{
0\ar[r]&K_{\rm sym} \ar[r]& \R[F_n]\ar[r]^{q_0} \ar[d]^i & \widehat{\mathcal B} \ar[d]^\iota\ar[r]&0\\
0\ar[r]&B \ar[r] & \R[F_n]\ar[r]^q & \widehat{\mathcal C}\ar[r]&0.
}\end{xy}\]
Note that the top row is exact by definition, whereas the bottom row is exact by Theorem~\ref{RelationsObvious}.(ii). We deduce that $i$ restricts to a linear map $j: K_{\rm sym} \to B$.

Now in order to show the desired inclusion $K_{\rm sym} \subset B+B_{\rm sym}$ it suffices to show that $j(K_{\rm sym}) \subset j(B)+j(B_{\rm sym})$ and that $\ker(j) \subset \ker(j|_{B+B_{\rm sym}})$. We will actually establish the stronger inclusions $j(K_{\rm sym}) \subset j(B)$ and $\ker(j) \subset B_{\rm sym}$. The latter inclusion is actually immediate from the description of $\ker(j)$ as
\[
\ker(j) = \{f \in K_{\rm sym} \mid j(f) = f+f^* = 0\}.
\]
It thus remains to show only that $j(K_{\rm sym})\subset j(B)$. We claim that
\begin{equation}\label{jB}
j(B) = \{b \in B\mid b^* =b\}.
\end{equation}
The inclusion $\subseteq$ follows from $(f+f^*)^* = f+f^*$. Conversely, if $b^* = b \in B\subset K_{\rm sym}$, then $j(b/2) = b/2 + b^*/2 = b$, which shows the opposite inclusion and proves \eqref{jB}.

Now if $f \in K_{\rm sym}$, then $j(f) \in B$ and $j(f)^* = (f+f^*)^* = f+f^*=j(f)$. We then deduce from \eqref{jB} that $j(f) \in j(B)$. This proves the remaining inclusion $j(K_{\rm sym}) \subset j(B)$.

We have thus shown that Part (ii) of Theorem~\ref{RelationsObvious} implies Part (iii). This concludes the proof of Theorem~\ref{RelationsObvious}.

\section{Bases for $\widehat{\mathcal C}(M_n)$, $\widehat{\mathcal C}(F_n)$ and $\widehat{\mathcal B}(F_n)$}\label{SecBasis}
\subsection{Pure bases and compatible bases}
The purpose of this section is to construct bases for each for the spaces $\widehat{\mathcal C}(M_n)$, $\widehat{\mathcal C}(F_n)$ and $\widehat{\mathcal B}(F_n)$ and thereby to establish Theorem~\ref{ThmBasis}. Given $L \geq 0$ we will denote by $\widehat{\mathcal C}(M_n)_L$ the image of ${\R}[M_n]_L$ in $\widehat{\mathcal C}(M_n)$. The spaces  $\widehat{\mathcal C}(F_n)_L$ and $\widehat{\mathcal B}(F_n)_L$ will be defined  similarly. We are going to relate bases of the spaces $\widehat{\mathcal C}(M_n)_L$ to bases of $\widehat{\mathcal C}(M_n)$, and similarly for $\widehat{\mathcal C}(F_n)$.

Concerning bases of $\widehat{\mathcal C}(M_n)_L$ we adapt the following language: A basis of $\widehat{\mathcal C}(M_n)_L$ is called \emph{pure} if its elements are of the form $[\rho_w]$ for some $w \in S^L$. Note that there are only finitely many pure bases for a given $L$. If $B_L$ is a basis for $\widehat{\mathcal C}(M_n)_L$, then the sequence of bases $(B_L)_{L \geq 0}$ is called \emph{compatible} if $B_L \cap \widehat{\mathcal C}(M_n)_{L-1} = B_{L-1}$.  Pure bases of $\widehat{\mathcal C}(F_n)$ and compatible sequences of bases of  $\widehat{\mathcal C}(F_n)$ are defined similarly. Note that a sequence of pure bases can never be compatible.

It follows from the left- and right-extension relations that every $[\rho_w]$ with $|w|=L$ can be written as a linear combination of $[\rho_v]$ for some $v$ of length $L+1$. This implies that
\[
\widehat{\mathcal C}(M_n)_0 \subset \widehat{\mathcal C}(M_n)_1 \subset \dots \subset \widehat{\mathcal C}(M_n)_L \subset \dots,
\]
whence $\widehat{\mathcal C}(M_n)$ is the ascending union
\[
\widehat{\mathcal C}(M_n) = \bigcup_{L \geq 0} \widehat{\mathcal C}(M_n)_L.
\]
Consequently, if $(B_L)_{L \geq 0}$ is a compatible sequence of bases of $\widehat{\mathcal C}(M_n)_L$, then $B:=\bigcup B_L$ defines a basis of $\widehat{\mathcal C}(M_n)$. Similarly, compatible sequences of bases for $\widehat{\mathcal C}(F_n)_L$ give rise to bases of $\widehat{\mathcal C}(F_n)$.

It turns out that pure bases of $\widehat{\mathcal C}(M_n)_L$ can be classified in graph theoretic terms. Given pure bases for each $\widehat{\mathcal C}(M_n)_L$ one can then easily modify them to obtain a compatible sequence of bases, and thereby a basis for  $\widehat{\mathcal C}(M_n)$. This will  be carried out in Subsection~\ref{BasisMonoid}, leading to a proof of Part (i) of Theorem \ref{ThmBasis}. The analogous constructions in the group case and in particular the proof of Part (ii) of Theorem \ref{ThmBasis} will be given in Subsection~\ref{BasisGroup}. Deducing Part (iii) of Theorem \ref{ThmBasis} from Part (ii) is essentially a triviality, since our basis for $\widehat{\mathcal C}(F_n)$ can easily be modified as to possess the necessary symmetries. We will give the details in Subsection~\ref{SecBasisQM}.

\subsection{A basis for $\widehat{\mathcal C}(M_n)$}\label{BasisMonoid}
Our first task is to classify pure bases for $\widehat{\mathcal C}(M_n)_L$. The only pure basis for $\widehat{\mathcal C}(M_n)_0$ is given by $B_0 = \{\rho_e\}$. We will now parametrize pure bases for $\widehat{\mathcal C}(M_n)_L$ for $L \geq 1$. Recall from Subsection~\ref{SubsecRelMonoids} that the $L$-th de Bruijn graph $\Gamma_L(S)$ over $S = \{a_1, \dots, a_n\}$ has vertices labelled by $S^{L-1}$ and edges labelled by $S^L$ where the edge labelled $w$ connects the two vertices by the subwords obtained from $w$ by deleting the first respectively last letter. The graph $\Gamma_L(M_n)$ is obtained from this graph by deleting loops and collapsing multiple edges to single edges. We observe:
\begin{prop}\label{BasisMonoidTrees} Let $L \geq 1$ and let $W$ be a set of words of length $L$ in $M_n$ of cardinality $|W| = (n-1)n^{L-1}+1$. Then the following are equivalent.
\begin{itemize}
\item[(i)] The set $B(W) := \{\rho_w\mid  w \in W\}$ is a pure basis of $\widehat{\mathcal C}(M_n)_L$.
\item[(ii)] The subgraph of $\Gamma_L(S)$ (or equivalently of $\Gamma_L(M_n)$) with vertices $S^{L-1}$ and edges labelled by $S^L\setminus W$ is connected.
\item[(iii)] The subgraph of $\Gamma_L(S)$ with vertices $S^{L-1}$ and edges labelled by $S^L\setminus W$ is a spanning tree of $\Gamma_L(S)$.
\end{itemize}
In particular, $\dim \widehat{\mathcal C}(M_n)_L = (n-1)n^{L-1}+1$ and pure bases of $\widehat{\mathcal C}(M_n)_L$ are in bijection with spanning trees of the de Bruijn graph $\Gamma_L(S)$.
\end{prop}
\begin{proof} For $L = 1$ the only pure basis of $\widehat{\mathcal C}(M_n)_L$ is $B(W)$ with $W = S$. This is in accordance with (ii) and (iii), since $\Gamma_1(S)$ has a single vertex. We may thus assume from now on that $L \geq 2$. By Remark \ref{RemarkB0} any relations between the $\rho_w$ are consequences of the basic relations $b_w$.  This can be expressed in terms of the matrix $A_L(M_n)$ given by \eqref{ALMn} as follows: Let us enumerate the words of length $L$ by $\{w_1, \dots, w_{n^L}\}$  and the words of length $L-1$ by $\{v_1, \dots, v_{n^{L-1}}\}$.
Then
\begin{eqnarray*}&&\sum_{i=1}^{n^L}\alpha_i[\rho_{w_i}] = 0\\
&\Leftrightarrow& \exists \lambda_j \; \forall i = 1, \dots, n^L:\; \alpha_i = \sum_{j=1}^{n^{L-1}}\lambda_ja_{ji},
\end{eqnarray*}
where $a_{ij}$ is the entry of $A_{L}(M_n)$ corresponding to the row $v_j$ and the column $w_j$. It follows that the set $\{\rho_w\mid w \in W\}$ is linearly independent if and only if the $n^{L-1} \times (n^{L-1}-1)$ submatrix $A_{W}$ of $A_L(M_n)$ formed by the columns corresponding to words in $S^L\setminus W$ has full rank $n^{L-1}-1$. The matrix $A_W$ has the same structure as $A$, i.e. every column contains at most two non-zero entries which are contained in in ${\pm 1}$ and sum up to $0$. Thus we can argue as in Step 1 of Subsection~\ref{SubsecRelMonoids} to conclude that $A_W$ has full rank if and only if the subgraph of $\Gamma_L(M_n) \subset \Gamma_L(S)$ with edges in $S^L\setminus W$ is connected. This shows the equivalence (i)$\Leftrightarrow$(ii) and also implies that $\dim \widehat{\mathcal C}(M_n)_L = (n-1)n^{L-1}+1$. Then the equivalence (ii)$\Leftrightarrow$(iii) is an immediate consequence of the fact that a graph with $k$ vertices and $k-1$ edges is connected if and only if it is a tree.
\end{proof}
For example let $S = \{a_1, a_2\}$ and consider the subset \[W := \{a_1^3, a_2a_1^2, a_2a_1a_2, a_2^2a_1a_2^2a_1, a_2^3\} \subset S^3.\] Then $S^3\setminus W = \{a_1^2a_2, a_1a_2a_1, a_1a_2^2\}$ corresponds to the following spanning tree of $\Gamma_3(S)$:
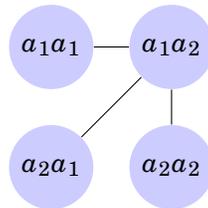
\begin{center}
\begin{tikzpicture}
  [scale=.8,auto=left,every node/.style={circle,fill=blue!20}]
  \node(0) at (3,1) {$a_2a_2$};
  \node (1) at (3,3) {$a_1a_2$};
  \node (2) at (1,1)  {$a_2a_1$};
  \node (3) at (1,3)  {$a_1a_1$};
  \foreach \from/\to in {0/1, 1/2, 1/3}
    \draw (\from) -- (\to);
\end{tikzpicture}
\end{center}
It follows that $B(W)$ is a pure basis of $\widehat{\mathcal C}(M_2)_3$. This example generalizes as follows:
\begin{cor}\label{PureBasisMonoid} Given $L \geq 0$ define $W(a_1; L)$ as follows: If $L = 0$, then $W(a_1; L) := \{e\}$. If $L > 0$, then
\[
W(a_1; L) := (S^L\setminus \{w \in S^L\mid w \in a_1S^{L-1}\})\cup\{a_1^L\}.
\]
Then $B(W(a_1; L)) :=\{\rho_w\mid w \in W(a_1; L)\}$ is a pure basis of $\widehat{\mathcal C}(M_n)_L$.
\end{cor}
\begin{proof} Let $w \in S^{L-1}$. We show that $w$ can be connected in $\Gamma_L(S)$ to $a_1^{L-1}$ by a path using only edges from $S^L \setminus W(a_1;L)$. For this we write $w = a_1^ks_1\dots s_{L-k-1}$ with $s_1 \neq a_1$. Then
\[
w =a_1^ks_1\dots s_{L-k-1} \sim a_1^{k+1}s_1\dots s_{L-k-2} \sim \dots \sim a_1^{L-2}s_1 \sim a_1^{L-1}
\]
is an admissible path.
\end{proof}
We can now modify these pure bases to obtain a family of compatible bases as follows:
\begin{cor}\label{CompatibleBasisMonoid} Let $j \in \{1, \dots, n\}$ and let $W_L$ denote the set of (possibly empty) words of length $\leq L$ not starting or ending in $a_j$. Then $B(W_L) := \{\rho_w\mid w\in W_L\}$ is a basis of $\widehat{\mathcal C}(W_n)_L$, and these bases are compatible.
\end{cor}
\begin{proof} By symmetry we may assume without loss of generality that $j=1$. Observe first that $|W_L| = \dim \widehat{\mathcal C}(M_n)_L$. For $L = 0$ this is obvious, and for $L \geq 1$ it follows from the formula $\dim \widehat{\mathcal C}(M_n)_L = (n-1)n^{L-1}+1$. In view of this observation it suffices to show that each element of the basis $B(W(a_1;L))$ can be expressed as a linear combination of elements in $B:=B(W_L)$. The case $L = 0$ is again obvious, so we may assume $L \geq 1$. In this case the set $W(a_1;L)$ can be written as the disjoint union
\[
W(a_1;L) = W_1 \cup W_2 \cup \{a_1^L\},
\]
where $W_1$ is the the set of words of length $L$ not starting and ending in $a_1$ and $W_2$ is the set of words of length $L$ ending, but not starting in $a_1$. If $w \in W_1$, then $\rho_w \in B$. If $w \in W_2$ then $w = va_1^k$, where $v$ does not end in $a_1$. If $k=1$, then
\[
\rho_w = \rho_v - \sum_{s \not \in \{w_{\rm fin}^{-1}, a_1\}} \rho_{vs}
\]
is contained in the span of $B$. For $k \geq 1$ we obtain $w \in B$ by induction on $k$, applying again the righ-extension relations. Finally, using again the right-extension relation and induction on $k$ one shows that $\rho_{a_1^k}$ is contained in the span of $B$ for all $k \leq L$. This shows that every element of $B({W(a_1;L)})$ is contained in the span of $B$ and finishes the proof.
\end{proof}
We deduce that the ascending union $\bigcup B(W_L)$ is a basis for $\widehat{\mathcal C}(M_n)$. This finishes the proof of Theorem \ref{ThmBasis}.(i).

\subsection{A basis for $\widehat{\mathcal C}(F_n)$}\label{BasisGroup}
We now modify the argument of the last subsection so that it works also for free groups instead of free monoids. Let $F_n$ be the free group with basis $S=\{a_1, \dots, a_n\}$ and let $\bar{S} = \{a_1, \dots, a_n, a_1^{-1}, \dots, a_n^{-1}\}$. We write $\bar{S}^{(L)} \subset \bar{S}^L$ for the subset of \emph{reduced} words of length $L$. With this notation, the Martin -- de Bruijn graph $\Gamma_L(\bar S)$ has vertex set $\bar{S}^{(L-1)}$ and edge set $\bar{S}^{(L)}$, where each edges connects the two vertices obtained by deleting the first respectively last letter. In complete analogy with Proposition~\ref{BasisMonoidTrees} one proves:
\begin{prop}\label{BasisGroupsTrees} Let $W$ be a set of reduced words of length $L \geq 2$ in $F_n$ of cardinality $|W| = 2n(2n-1)^{L-2}(2n-2) + 1$. Then the following are equivalent.
\begin{itemize}
\item[(i)] The set $B(W) := \{\rho_w\mid  w \in W\}$ is a pure basis of $\widehat{\mathcal C}(F_n)_L$.
\item[(ii)] The subgraph of $\Gamma_L(\bar S)$ with vertices $\bar{S}^{(L-1)}$ and edges labelled by $\bar{S}^{(L)}\setminus W$ is connected.
\item[(iii)] The subgraph of $\Gamma_L(\bar S)$ with vertices $\bar{S}^{(L-1)}$ and edges labelled by $\bar{S}^{(L)}\setminus W$ is a spanning tree of $\Gamma_L(\bar S)$.
\end{itemize}
In particular, $\dim \widehat{\mathcal C}(F_n)_L = 2n(2n-1)^{L-2}(2n-2) + 1$ and pure bases of $\widehat{\mathcal C}(F_n)_L$ are in bijection with spanning trees of the Martin -- de Bruijn graph $\Gamma_L(S)$.\qed
\end{prop}
There are two exceptional cases $L = 0$ and $L=1$. The unique pure basis of $\widehat{\mathcal C}(F_n)_0$ is given by $\{\delta_e\}$, and the unique pure basis of $\widehat{\mathcal C}(F_n)_1$ is given by $B(\bar S)$. Both statements are immediate from our earlier computation $\dim K(F_n)_0 = \dim K(F_n)_1 = 0$. In general we can choose the following pure basis:
\begin{cor}\label{PureBasisGroup} Given $L \geq 0$ define $W(a_1; L)$ as follows: If $L = 0$, then $W(a_1; 0) := \{e\}$, if $L=1$, then $W(a_1; 1) := \bar S$ and if $L \geq 2$, then
\[W(a_1; L):= (\bar{S}^{(L)}\setminus \{w \in \bar{S}^{(L)}\mid w \text{ starts with }  a_1 \text{ or }a_2a_1^{-1}\})\cup\{a_1^L\}.\] Then $B(W(a_1; L)) :=\{\rho_w\mid w \in W(a_1; L)\}$ is a pure basis of $\widehat{\mathcal C}(F_n)_L$.
\end{cor}
\begin{proof} For $L \leq 1$ there is nothing to show. For $L \geq 2$ we have to show that every vertex $w \in \bar{S}^{(L-1)}$ can be connected to $a_1^{L-1}$ using only edges starting with $a_1$ (but not equal to $a_1^L$) and $a_2a_1^{-1}$. If $w$ does not start with $a_1^{-1}$ we can argue as in the monoid case, using only edges of the first kind. If $w= a_1^{-1}s_2\dots s_{L-1}$ then an admissible path to $a_1^{L-1}$ is given by
\begin{eqnarray*}
a_1^{-1}s_2\dots s_{L-1} &\sim& a_2a_1^{-1}s_2\dots s_{L-2} \sim a_1a_2a_1^{-1}s_2\dots s_{L-3} \sim \dots \\ &\sim& a_1^{L-3}a_2a_1^{-1} \sim a_1^{L-2}a_2^{-1}\sim a_1^{L-1}.\quad \quad \quad \quad\quad \quad \quad \quad\quad \quad \quad\quad\qedhere
\end{eqnarray*}
\end{proof}
Again it is easy to pass from a family of pure bases to a family of compatible bases:
\begin{cor}\label{CompatibleBasisGroup} Given $L \geq 1$ let $W_L'$ denote the set of words of length $\leq L$ (including the empty word) not starting in $a_1$ or $a_2a_1^{-1}$ and not ending in $a_1^{-1}$ or $a_1a_2^{-1}$. Let $W_1 := W_1' \cup \{a_1^{-1}\}$ and $W_L := W_L'$ otherwise. Then $B(W_L) := \{\rho_w\mid w\in W_L\}$ is a basis of $\widehat{\mathcal C}(W_n)_L$, and these bases are compatible.
\end{cor}
\begin{proof} Again we have to produce every element $\rho_w$ in $B(a_1; L)$ as a linear combination of elements in $B:=B(W_L)$. If $w$ does not end in $a_1^{-1}$ or $a_1a_2^{-1}$ then $\rho_w \in B$. Otherwise we can apply right-extension relations and argue inductively just as in the proof of Proposition~\ref{CompatibleBasisMonoid}.
\end{proof}
Now Theorem \ref{ThmBasis}.(ii) follows.

\subsection{A basis for $\mathcal B(F_n)$}\label{SecBasisQM}
We conclude this section by pointing out that Part (iii) of Theorem \ref{ThmBasis} follows from Part (ii). The proof is based on the following simple observation:
\begin{lem}\label{BasisQM}
Let $W \subset \mathcal F_n$ be any set of reduced words such that $B(W) := \{\rho_w\,|\,w \in W\}$ is a basis of $\widehat{\mathcal C}(F_n)$ and
such that $w^{-1} \in W$ for all $w \in W$.  Let $W^+$ be a subset of $W$ which intersects each of the sets $\{w,w^{-1}\}$ in a single element. Then $B(W^+) :=\{\phi_w\mid w \in W^+\}$ is a basis of $\widehat{\mathcal B}(F_n)$.\qed
\end{lem}
\begin{proof} This is immediate from the fact that $\widehat{\mathcal B}(F_n) \subset \widehat{\mathcal C}(F_n)$ is the fixed point set of the linear involution mapping $\rho_w$ to $-\rho_{w^{-1}}$.
\end{proof}
The basis of $\widehat{\mathcal C}(F_n)$ constructed in Corollary \ref{CompatibleBasisGroup} does not satisfy the assumptions of the lemma. However, we can modify it as follows. Let $W_0 := \bigcup W_L$, where $W_L$ is defined as in Corollary \ref{PureBasisGroup} and let $W := (W_0 \setminus \{e\}) \bigcup \{a_1\}$. Then $B(W)$ is another basis for $\widehat{\mathcal C}(F_n)$ and $W$ satisfies the assumptions of Lemma \ref{BasisQM}. This shows that Part (iii) of Theorem \ref{ThmBasis} follows indeed from Part (ii), and thereby finishes the proof of the theorem.

\section{Sums of counting functions and weighted trees}\label{SecTrees}

\subsection{Representing sums of counting functions by weighted trees}\label{SecTreeRepresentation}

The goal of this subsection is to provide a graphical representation for sums of counting functions on free groups and monoids. This will help us to visualize certain operations on counting functions and allow us to decide whether a given sum of counting functions is bounded, i.e. represents the trivial element in $\widehat{\mathcal C}(M_n)$ or $\widehat{\mathcal C}(F_n)$. We start by discussing the case of monoids.

Denote by $T(M_n)$, or $T_n$ for short, the right-Cayley tree of $M_n$ with respect to $S$, i.e. the vertex set of $V(T_n)$ of $T_n$ is given by $V(T_n) = M_n$ and $w \in M_n$ is connected by an edge to $wa_j$ for each $j=1, \dots, n$. We think of $T_n$ as a coloured rooted tree with root $e$, where edges are coloured by the generating set $S$. We define the \emph{depth} of a vertex $w$ as the word length of $w$ or, equivalently, the distance of the vertex from the root. By a \emph{weight} on $T_n$ we mean a finitely supported real-valued function $\alpha$ from the vertices of $T_n$ to $\R$. We can visualize the pair $(T_n, \alpha)$ by drawing the finite subtree of $T_n$ spanned by the union of the support of $\alpha$ together with the root and labelling every vertex $w$ by $\alpha(w)$. The \emph{depth} of the weighted tree $(T_n, \alpha)$ is defined as $-\infty$ if $\alpha = 0$, and as
\begin{equation}\label{depthtree}
 L(T_n, \alpha) := \max\{|w|\mid \alpha(w) \neq 0\}
\end{equation}
otherwise. The following picture shows an example of a weighted tree of depth $2$ for $n=3$:
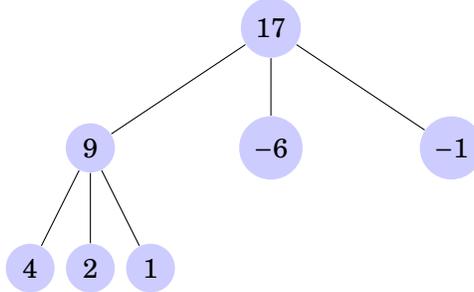
\begin{center}
\begin{tikzpicture}
  [scale=.8,auto=left,every node/.style={circle,fill=blue!20}]
  \node(0) at (15,5) {$17$};
  \node (1) at (12,3) {$9$};
  \node (2) at (15,3)  {$-6$};
  \node (3) at (18,3)  {$-1$};
   \node (4) at (11,1) {$4$};
  \node (5) at (12,1)  {$2$};
  \node (6) at (13,1)  {$1$};
  \foreach \from/\to in {0/1, 0/2, 0/3, 1/4, 1/5, 1/6} 
    \draw (\from) -- (\to);
\end{tikzpicture}
\end{center}
Given such a weighted tree $(T_n, \alpha)$ we define the associated sum of counting functions by
\[
c(T_n, \alpha) := \sum_{w \in M_n} \alpha(w) \cdot \rho_{w} \in \mathcal C(M_n).
\]
We then say that the weighted trees $(T_n, \alpha)$ and $(T_n, \alpha')$ are \emph{equivalent} if $[c(T_n, \alpha)] = [c(T_n, \alpha)' ] \in \widehat{\mathcal C}(M_n)$. Our goal is to understand geometrically what it means for two weighted trees to be equivalent. Since there is an obvious geometric way to subtract weighted trees, it suffices to understand geometrically whether
$[c(T_n, \alpha)] = [0]$ for a given weighted tree  $(T_n, \alpha)$.

Everything we said so far carries over verbatim to the case of a free group $F_n$, if we replace the tree $T(M_n)$ by the right-Cayley tree $T_n := T(F_n)$ of $F_n$ with respect to the generating set $S$. As in the monoid case, we also would like to understand in the group case the geometric meaning of the condition $[c(T_n, \alpha)] = [0]$.

\subsection{Operations on weighted trees}\label{SecTreeOperations}
Throughout this subsection let $T_n$ be either $T(M_n)$ or $T(F_n)$. We describe some operations which transform a weighted tree $(T_n, \alpha)$ into an equivalent weighted tree $(T_n, \alpha')$.

Let $V(T_n)$ be the vertex set of $T_n$. Given $w \in V(T_n)$ of depth $\geq 1$, we refer to the vertices on the geodesic between $e$ in $w$ (including $e$, but excluding $w$) as the \emph{ancestors} of $w$. The unique ancestor $v={\rm Fa}(w)$ of $w$ adjacent to $w$ is called its \emph{father} and the vertices with the same father as $w$ are called its \emph{brothers}. Their collection is called the \emph{brotherhood} of $w$ and denoted by ${\rm Br}(w)$. We say that a brotherhood $B$ is a \emph{constant brotherhood} with respect to $\alpha$ if $\alpha|_B$ is constant, and a \emph{non-constant brotherhood} otherwise. We also define the \emph{depth of a brotherhood} $B$ as the depth of any of its members and denote by ${\rm Fa}(B)$ the common father of the brotherhood. Note that by definition the depth of a brotherhood is $\geq 1$.

If two vertices $u, v \in M_n$ have the same depth $|u|=|v|=L \geq 1$ and differ only by the first letter, then we say that they are \emph{related} and write $u \smile v$. In this case we also say that the brotherhood $B_1:= \mathrm{Br}(u)$ and $B_2 := \mathrm{Br}(v)$ are related and write $B_1\smile B_2$. If $B_1 \smile B_2$, then there is a unique bijection $\iota_{B_1, B_2}: B_1 \to B_2$ with the property that $\iota_{B_1, B_2}(w) \smile w$.

In the monoid case, each brotherhood has exactly $n$ elements, and every brotherhood of depth $\geq 2$ has exactly $n$ related brotherhoods including itself. In the group case, every brotherhood of depth $\geq 2$ has $2n-1$ elements, and there is a unique exceptional brotherhood of depth $1$ containing $2n$ elements. In this case, every brotherhood of depth $\geq 2$ has $2n-1$ related brotherhoods including itself. We now introduce the following two types of operations.

Firstly, let $B$ be a brotherhood of depth $\geq 1$ with father $w$. In the monoid case, let $s \in S$, and in the group case let $s \in \bar S \setminus\{w_{\rm fin}^{-1}\}$, where $w_{\rm fin}$ denotes the last letter of $w$. Then the \emph{partial reduction} of $B$ along $s$ is the operation $(T_n, \alpha) \mapsto {\rm Red}_{B, s}(T_n, \alpha) := (T_n, \alpha')$, where $\alpha' \in \R[M_n]$ is given as follows: Let $v_0 \in B$ be the unique element with final letter $s$. Then $\alpha'(w):= \alpha(w) + \alpha(v_0)$, $\alpha'(v) := \alpha(v)-\alpha(v_0)$ for all $v \in B$ and $\alpha'(v) := \alpha(v)$ in all other cases. Then, by the right extension relations $r_w$, the operation ${\rm Red}_{B, s}$, transforms every weighted tree into an equivalent weighted tree. Note that $\alpha'$ differs from $\alpha$ only along the brotherhood $B$ and its father.

A special case appears if $B$ is a constant family with respect to $\alpha$. In this case all the partial reductions ${\rm Red}_{B, a_j}$ have the same effect on $(T_n, \alpha)$, and we have ${\rm Red}_{B, a_j}\alpha|_B \equiv 0$. In this case we refer to ${\rm Red}_{B, a_j}(T_n, \alpha)$ as the \emph{(complete) reduction} of $(T_n, \alpha)$ along $B$. The following pictures show an effect of two subsequent reductions:\\
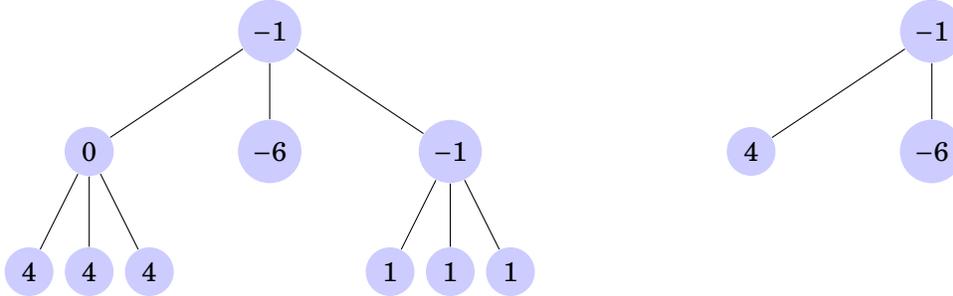
\begin{tikzpicture}
  [scale=.8,auto=left,every node/.style={circle,fill=blue!20}]
  \node(0) at (5,5) {$-1$};
  \node (1) at (2,3) {$0$};
  \node (2) at (5,3)  {$-6$};
  \node (3) at (8,3)  {$-1$};
   \node (4) at (1,1) {$4$};
  \node (5) at (2,1)  {$4$};
  \node (6) at (3,1)  {$4$};
     \node (7) at (7,1) {$1$};
  \node (8) at (8,1)  {$1$};
  \node (9) at (9,1)  {$1$};
   \node(0a) at (16,5) {$-1$};
  \node (1a) at (13,3) {$4$};
  \node (2a) at (16,3)  {$-6$};
  \foreach \from/\to in {0/1, 0/2, 0/3, 1/4, 1/5, 1/6, 3/7, 3/8, 3/9, 0a/1a, 0a/2a} 
    \draw (\from) -- (\to);
\end{tikzpicture}\\

We now define a second operation called \emph{transfer} which corresponds to the left-extension relation. Since we chose to work with \emph{right}-Cayley graphs, the geometric meaning of this operation is less natural.

To define transfer, let $B$ be a brotherhood of depth $L \geq 2$. 
Then the \emph{transfer of $B$} is the operation $(T_n, \alpha)\mapsto {\rm Tr}_{B}(T_n, \alpha) := (T_n, \alpha')$, where $\alpha'$ is given as follows:  If $v \in B$ and $v'$ is obtained from $v$ by deleting the first letter, then $\alpha'(v) = 0$ and $\alpha'(v') = \alpha(v')+\alpha(v)$. Moreover, 
if $v$ is contained in a brotherhood $B'$ related to $B$, then $\alpha'(v):=\alpha(v) -\alpha(\iota_{B',B}(v))$. Finally, $\alpha'(v) := \alpha(v)$ for all other vertices $v$.

The following picture shows the effect of transfer applied to the brotherhood labelled $1,2,3$.
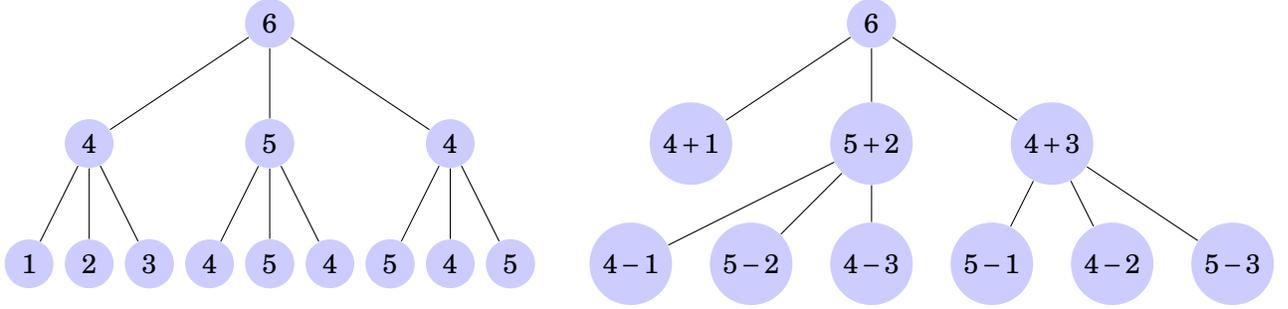
\begin{center}
\begin{tikzpicture}
  [scale=.8,auto=left,every node/.style={circle,fill=blue!20}]
  \node(0) at (5,5) {$6$};
  \node (1) at (2,3) {$4$};
  \node (2) at (5,3)  {$5$};
  \node (3) at (8,3)  {$4$};
   \node (4) at (1,1) {$1$};
  \node (5) at (2,1)  {$2$};
  \node (6) at (3,1)  {$3$};
  \node(20) at (4,1) {$4$};
  \node(21)at (5,1)  {$5$};
  \node(22)at (6,1)   {$4$};
    \node (7) at (7,1) {$5$};
  \node (8) at (8,1)  {$4$};
  \node (9) at (9,1)  {$5$};
    \node(10) at (15,5) {$6$};
  \node (11) at (12,3) {$4+1$};
  \node (12) at (15,3)  {$5+2$};
  \node (13) at (18,3)  {$4+3$};
   \node (14) at (11,1) {$4-1$};
  \node (15) at (13,1)  {$5-2$};
  \node (16) at (15,1)  {$4-3$};
    \node (17) at (17,1) {$5-1$};
  \node (18) at (19,1)  {$4-2$};
  \node (19) at (21,1)  {$5-3$};
  \foreach \from/\to in {0/1, 0/2, 0/3, 1/4, 1/5, 1/6, 2/20, 2/21, 2/22, 3/7, 3/8, 3/9, 10/11, 10/12, 10/13, 13/17, 13/18, 13/19, 12/14, 12/15, 12/16}
    \draw (\from) -- (\to);
\end{tikzpicture}
\end{center}

By definition, transfer has the following properties: ${\rm Tr}_B$ maps every weighted tree to an equivalent one. If changes the values of $\alpha$ only on (certain) elements of depth $L-1$ and on those elements of depth $L$ which are related to a member of $B$. Moreover, ${\rm Tr}_B\alpha|_B \equiv 0$. We emphasize that we will only apply transfer to brotherhoods of depth at least $2$.

\subsection{Unbalanced weighted trees}
We now define a special class of weighted trees; as in the last subsection $T_n$ denotes either $T(M_n)$ or $T(F_n)$.
\begin{definition} A weighted tree $(T_n, \alpha)$ of depth $L\geq 2$ is called \emph{unbalanced} if there exist two related brotherhoods $B_1 \smile B_2$ of depth $L$ such that $\alpha|_{B_1} \equiv 0$ and $B_2$ is non-constant with respect to $\alpha$.
\end{definition}
 It is easy to see from the picture, whether a given weighted tree is unbalanced. For example, the weighted tree given in Subsection \ref{SecTreeRepresentation} is unbalanced.
\begin{thm}\label{Unbalanced} Every unbalanced weighted tree represents a non-trivial element in $\widehat{\mathcal C}(M_n)$ or $\widehat{\mathcal C}(F_n)$.
\end{thm}
We discuss the proof separately in the monoid case and in the group case:
\begin{proof}[Proof of Theorem \ref{Unbalanced} in the monoid case]

Using the operations defined in Subsection \ref{SecTreeOperations} we will transform $(T_n, \alpha)$ into an equivalent weighted tree, which is non-trivial for obvious reasons. Let $B_1 \smile B_2$ be brotherhoods of depth $L:={\rm depth}(T_n, \alpha)$ such that $\alpha|_{B_1} \equiv 0$ and $B_2$ is non-constant with respect to $\alpha$, and let $a_i$, respectively $a_j$ be the first letters of ${\rm Fa}(B_1)$ and ${\rm Fa}(B_2)$.

Firstly, we transfer all brotherhoods of depth $L$ in the subtree $a_iM_n$ except for the brotherhood $B_1$. The weights of $B_2$ remain the same, because only the transfer of $B_1$ could affect $B_2$, but $B_1$ was not transferred. After these transfers, all coefficients in the level $L$ in $a_iM_n$ are equal to $0$. Secondly, we perform a partial reduction of all brotherhoods in the level $L$ with respect to the ending $a_i$. Since $B_2$ was non-constant, it remains non-constant under these partial reductions. Since all the brotherhoods in $a_iM_n$ had coefficients $0$, they also remain $0$.

Now we repeat the same procedure in the levels $l=L-1, L-2, \ldots, 2$. Namely, first we transfer all brotherhoods having depth $l$ from the subtree $a_iM_n$. This affects values in levels $\leq l$ and makes all coefficients in the subtree $a_iM_n$ in the levels $l,\ldots, L$ equal to $0$. Secondly, we apply partial reduction with respect to the letter $a_i$ in all brotherhoods (except those of the subtree $a_iM_n$) of level $l$. This affects level $l-1$ and makes all coefficients in level $l$ of words ending with $a_i$ equal to $0$. The brotherhood $B_2$ remains non-constant throughout.

Finally we reach a  weighted tree $(T_n, \alpha')$ equivalent to $(T_n, \alpha)$ with the following properties: If $w$ is any word of depth $\geq 2$ which starts of ends in $a_i$, then $\alpha'(w) = 0$. Morover, $B_2$ is non-constant with respect to $\alpha'$. We now do one final reduction of the brotherhood ${\rm Br}(a_1)$ with respect to $a_i$ to obtain yet another equivalent weighted tree $(T_n, \alpha'')$. Now $\alpha''$ vanishes on all words $w$ starting or ending in $a_i$, but is not equal to $0$ (since $B_2$ is non-constant). It then follows from Corollary \ref{CompatibleBasisMonoid}  that $[\alpha]=[\alpha''] \neq 0$.
\end{proof}
The strategy of the proof can be described as clearing out all coefficients of vertices starting or ending in $a_i$. This strategy works because of Corollary \ref{CompatibleBasisMonoid}. In the group case we have to replace Corollary \ref{CompatibleBasisMonoid} by Corollary \ref{CompatibleBasisGroup}. We therefore have to clear out all coefficients of vertices
starting in $a_1$ or $a_2a_1^{-1}$ or ending in $a_1^{-1}$ or $a_1a_2^{-1}$, except for $a_1^{-1}$. This is slightly more complicated than in the monoid case, but ultimately works the same way.

\begin{proof}[Proof of Theorem \ref{Unbalanced} in the group case]
By assumption we have two brotherhoods $B_1 \smile B_2$ of depth $L:={\rm depth}(T_n, \alpha)$ such that $\alpha|_{B_1} \equiv 0$ and $B_2$ is non-constant with respect to $\alpha$. Without loss of generality we may assume that the first letters of ${\rm Fa}(B_1)$ and ${\rm Fa}(B_2)$ respectively are $a_1$ and $a_2$. Note that the second letter of ${\rm Fa}(B_1)$, and hence also of ${\rm Fa}(B_2)$ canot be $a_1^{-1}$.

Let $\alpha'$ be obtained from $\alpha$ by applying the following operations: Firstly, apply transfer to all brotherhoods of depth $L$ in the subtree of reduced words starting from $a_1$. Secondly, apply transfer to all brotherhoods of depth $L$  in the subtree of reduced words starting from $a_2a_1^{-1}$. Finally, apply a partial reduction ${\rm Red}_{B, s}$ for every non-zero brotherhood of depth $L$, where $s$ depends on the last letter of ${\rm Fa}(B)$. If this last letter happens to be $a_1$, then we choose $s:=a_2^{-1}$, otherwise we choose $s:=a_1^{-1}$.

We now consider the values of $\alpha'$ on words $w$ of depth $L$. Assume first that $w$ starts with $a_1$. Then $\alpha'(w) = 0$ after the first transfer step. The second transfer step only transfers into words whose second letter is $a_1^{-1}$, hence does not change $\alpha'(w)$. Since also the partial reductions do not influence the family of $w$, we get $\alpha'(w) = 0$. Similarly, if $w$ starts with $a_2a_1^{-1}$, then $\alpha'(w) = 0$. Finally, if $w$ ends with $a_1^{-1}$ or with $a_1a_2^{-1}$, then $\alpha'(w)=0$, since $w$ gets cleared in the partial reduction steps. On the other hand, we claim that the brotherhood $B_2$ does not get cleared completely and in fact remains non-constant for $\alpha'$. The unique brotherhood with initial letter $a_1$ related to $B_2$ is $B_1$, and this one does not transfer anything over in the first step, since $\alpha|_{B_1}\equiv 0$. Since the second letter of ${\rm Fa}(B_2)$ is not $a_1^{-1}$, the values of $\alpha'$ on $B_2$ also do not change in the second transfer step. In the partial reduction step, the value of $\alpha'$ on $B_2$ is changed, however, a partial reduction cannot turn a non-constant family into a non-constant family.

We can now repeat the same procedure on levels $l=L-1, L-2, \ldots, 2$. Ultimately we end up with a weighted tree $(T_n, \alpha')$ equivalent to $(T_n, \alpha)$ with the following properties: If $w$ is any word of depth $\geq 2$ which starts with $a_1$  or $a_2a_1^{-1}$ or end with $a_1^{-1}$ or $a_1a_2^{-1}$, then $\alpha'(w) = 0$. Moreover, there exists $w_0 \in B_2$ with $\alpha(w_0) \neq 0$.

Now $\alpha'$ is equivalent to $\alpha''$ as given by $\alpha''(s) := \alpha'(s)-\alpha'(a_1)$ for all $s \in \bar S$, $\alpha''(e) := \alpha'(e)+\alpha'(a_1)$ and $\alpha''(w) = \alpha'(w)$ for all words $w$ of length $\geq 2$. Moreover, if $W_L$ is defined as in Corollary \ref{CompatibleBasisGroup}, then $\alpha''(w) = 0$ for all $w \not \in \bigcup W_L$ and $\alpha''(w_0) \neq 0$. It then follows from Corollary \ref{CompatibleBasisGroup} that $\alpha''$ and hence $\alpha$ does not represent the $0$ function.
\end{proof}

\section{Deciding boundedness for sums of counting functions}\label{SecAlgorithmsNaive}

We now present an algorithm to decide whether a given sum of counting functions is bounded, which is based on Theorem \ref{Unbalanced}. The basic strategy is as follows: Represent the given function by a weighted tree, and try to transform this weighted tree either into the empty weighted tree or an unbalanced weighted tree using the operations discussed in Subsection \ref{SecTreeOperations}. If you reach the empty tree, then the initial function was bounded, and if you reach an unbalanced weighted tree, then the function was unbounded. To obtain an actual algorithm, we have to ensure that we either reach an unbalanced tree or the empty tree within a finite number of operations.

In the monoid case, the algorithm looks as follows:

\noindent\hrulefill
\begin{center}
Algorithm \textsc{Decide Triviality in $\widehat{\mathcal C}(M_n)$}
\end{center}
\noindent \textsc{Input}: Weighted tree $(T_n, \alpha)$.

\noindent \textsc{Output}: \emph{Trivial} or \emph{Non-trivial} according to whether $[c(T_n,\alpha)] = [0]$ or not.

\begin{enumerate}[\textsc{Step }1]
\item Let $\alpha' := \alpha$, and let $l$ be the depth of $\alpha'$.
\item While $l \geq 2$ repeat the following steps to the tree $\alpha'$:
\begin{enumerate}
\item Transfer all brotherhoods which start with $a_1$ and have depth $l$.
\item Reduce all constant brotherhoods of depth $l$.
\item If the depth of $\alpha'$ is still $l$, return \emph{non-trivial} and stop the algorithm. Otherwise, replace $l$ by the new length of $\alpha'$.
\end{enumerate}
\item If $\alpha'(a_i) = -\alpha'(e)$ for each $i=1, \dots, n$, then return \emph{trivial}, otherwise return \emph{non-trivial}.
\end{enumerate}
\noindent \hrulefill

Since $l$ decreases by at least one in each iteration of \textsc{Step 2}, the algorithm terminates. Let us verify correctness of the algorithm: Since $\alpha'$ is equivalent to $\alpha$ at all stages, if the algorithm returns \emph{trivial}, then indeed $[c(T_n,\alpha)] = [0]$. Conversely, if the algorithm returns \emph{non-trivial} in \textsc{Step 2c}, then we have reached an unbalanced tree, so indeed $[c(T_n,\alpha)] \neq [0]$ by Theorem \ref{Unbalanced}. (The theorem applies, since every brotherhood is related to a brotherhood with initial letter $a_1$.) Also, if the algorithm returns \emph{non-trivial} in \textsc{Step 3}, then it follows from Corollary \ref{PureBasisMonoid} applied to $L =1$ that $[\alpha]\neq 0$. Thus the algorithm works correctly.

Almost the same algorithm works in the group case. The main difference appears in \textsc{Step 2}(a), where we have also to clear elements starting with $a_2a_1^{-1}$. This is because not every brotherhood is related to a brotherhood starting from $a_1$, but every brotherhood is related to a brotherhood starting from either $a_1$ (if its second letter is not $a_1^{-1}$) or $a_2a_1^{-1}$ (otherwise). Also, in  \textsc{Step 3} we have to take the inverses of the generators into account.

This leads to the following algorithm:

\bigskip

\noindent\hrulefill
\begin{center}
Algorithm \textsc{Decide Triviality in $\widehat{\mathcal C}(F_n)$}
\end{center}
\noindent \textsc{Input}: Weighted tree $(T_n, \alpha)$.

\noindent \textsc{Output}: \emph{Trivial} or \emph{Non-trivial} according to whether $[c(T_n,\alpha)] = [0]$ or not.

\begin{enumerate}[\textsc{Step }1]
\item Let $\alpha' := \alpha$, and let $l$ be the depth of $\alpha'$.
\item While $l \geq 2$ repeat the following steps to the tree $\alpha'$:
\begin{enumerate}
\item Transfer all brotherhoods which start with $a_1$ or or $a_2a_1^{-1}$ and have depth $l$.
\item Reduce all constant brotherhoods of depth $l$.
\item If the depth of $\alpha'$ is still $l$, return \emph{non-trivial} and stop the algorithm. Otherwise, replace $l$ by the new length of $\alpha'$.
\end{enumerate}
\item If $\alpha'(a_i)=\alpha'(a_i^{-1})= -\alpha'(e)$ for each $i=1, \dots, n$, then return \emph{trivial}, otherwise return \emph{non-trivial}.
\end{enumerate}
\noindent \hrulefill

The proof for termination and correctness is as in the monoid case. The two algorithms presented here suffer from two defects. 

Firstly, they are not quite optimal as far as their runtime is concerned. The main problem is that in the transfer step \textsc{Step 2}(a), big coefficients in the transferred brotherhood may generate big coefficients in related brotherhoods. For this reason, it is not optimal to always apply  transfer to the brotherhoods starting with $a_1$ (or $a_2a_1^{-1}$ in the group case). By a more cleverly chosen (but more complicated) combination of transfer and (partial) reduction steps once can in fact achieve that the input list of depth $L$ is converted into a list of depth $L-1$ of smaller size. This also reduces the runtime.

Secondly, the ``algorithms'' above aren't actually algorithms in the formal sense, since we do not specify how exactly the data is stored and how exactly addition and comparison of numbers are implemented. Without specifying these details, one cannot even start to discuss the runtime of our algorithms. Moreover, it turns out that a detailed runtime analysis of our algorithms requires some considerations in complexity theory. Since this analysis is of a somewhat different flavour than the topics covered in the present article, we will discuss the optimized algorithms in a separate article \cite{Sequel}.

\appendix

\section{Homogenizations of counting functions}
In this appendix we discuss certain classes of homogeneous functions related to counting functions.

\begin{definition} Let $M$ be a monoid. A function $f: M \to \R$ is called \emph{homogeneous} if $f(g^n) = n \cdot f(g)$ for all $g \in M$ and $n \geq 0$. It is called \emph{homogenizable} if for every $g \in M$ the limit
\[
\widehat{f}(g) = \lim_{n \to \infty} \frac{f(g^n)}{n}
\]
exists. In this case, $\widehat{f}: M \to \R$ is called the \emph{homogenization} of $f$.
\end{definition}

By basis properties of limits, the homogenizable functions form a real vector space, and homogenization defines a linear endomorphism of this vector space, whose image is given by the subspace of homogeneous functions. Moreover, if two homogenizable functions are at bounded distance, then their homogenizations coincide.

It is well-known that quasimorphisms are homogenizable (see e.g. \cite{scl}). Moreover, if $f$ is a quasimorphism, then its homogenization $\widehat{f}$ can be characterized as the unique homogeneous function at bounded distance from $f$. Moreover, two quasimorphisms are at bounded distance if and only if their homogenizations coincide.

In this appendix we will show (following closely an argument from \cite{Grigorchuk} for counting quasimorphisms) that counting functions on free monoids and groups are also homogenizable. However, in the group case it is \emph{not} true that a counting function is at bounded distance from its homogenization. Consequently, some standard arguments from the theory of homogeneous quasimorphisms do not carry over to the setting of counting functions. This caveat is the reason why we work out a couple of otherwise standard arguments in detail.

From now on let $S = \{a_1, \dots, a_n\}$ be a set of cardinality $n$. We denote by $M_n$ and $F_n$ respectively the free monoid and free group with basis $S$. In analogy with the subword relation discussed in the introduction we can also introduce a cyclic subword relation as follows. Informally, if $v , w \in M_n$ we say that $v$ is a \emph{cyclic subword} of $w$ if $v$ can be read off by running along (possibly several times) the cyclic word obtained by closing up $w$ (i.e. writing $w$ along a circle). Thus e.g. $a_1a_2$ is a cyclic subword of $a_2a_1a_3a_1$, but also $a_1^3$ is a cyclic subword of $a_1^2$. To define this more formally, we introduce the following notation.

Given positive integers $l.m$ we denote by $[l]_m$ the unique number in $\{1, \dots, m\}$ which is congruent to $l$ modulo $m$.  Then $v = s_1\cdots s_l \in M_n$ is a cyclic subword of $w = r_1\cdots r_m\in M_n$ if there exists $j \in \{1, \dots, m\}$ such that
\begin{equation}\label{cyclicsubword}
s_i = r_{[j+i]_{m}} \quad \text{ for all } i = 1, \dots, l.
\end{equation}
Given $v \in M_n$ we define the \emph{cyclic counting function} $\widehat{\rho_v}: M_n \to \Z$ as follows: $\widehat{\rho_v}(w)$ counts the cyclic occurrences of $v$ in $w$, i.e. if $v = s_1\cdots s_l$ and $w = r_1\cdots r_m\in M_n$, then $\widehat{\rho_v}(w)$ is the number of $j \in \{1, \dots, m\}$ such that \eqref{cyclicsubword} holds.
\begin{lem}\label{HomogenizationMonoid} The counting function $\rho_v: M_n \to \Z$ is homogenizable, and its homogenization is given by the cyclic counting function $\widehat{\rho_v}: M_n \to \Z$.
\end{lem}
\begin{proof} The cyclic counting function is obviously homogeneous. Moreover, if \eqref{cyclicsubword} holds for some $j \in \{1, \dots, m\}$ and $j\leq m-l$, then \eqref{subword} holds for the same $j$. It follows that $\|\widehat{\rho_v} - \rho_v\|_\infty \leq |v|_S$. Thus for all $w \in M_n$,
\[
\lim_{n \to \infty} \frac{\rho_v(w^n)}{n} =  \lim_{n \to \infty} \left(\frac{\widehat{\rho_v}(w^n)}{n} + \frac{\rho_v(w^n)-\widehat{\rho_v}(w^n)}{n}\right) =  \lim_{n \to \infty}\frac{\widehat{\rho_v}(w^n)}{n} = \lim_{n \to \infty}\frac{n \cdot \widehat{\rho_v}(w)}{n} = \widehat{\rho_v}(w).\qedhere
 \]
\end{proof}
Note that in the monoid case we have $\|\widehat{\rho_v} - \rho_v\|_\infty  < \infty$. We will see in Example \ref{ExampleHomogenization} that the corresponding statement fails in the group case. The reason for this failure is given by cyclic cancellations, as we explain next.

Recall that a reduced word $w \in F_n$ is called \emph{cyclically reduced} if its initial letter is not the inverse of its final letter. In this case we can close up $w$ and obtain a reduced cyclic word. Every reduced word $w \in F_n$ is conjugate to a cyclically reduced (and reduced) word $w_0$ (sometimes called the  \emph{cyclic reduction} of $w$), which is unique up to cyclic permutation. In particular, the cyclic word obtained by closing $w_0$ depends only on $w$.  Given a reduced word $v \in F_n$ we define a \emph{cyclic counting function} $\widehat{\rho_v}: F_n \to \Z$ as follows. Given a reduced word $w \in F_n$, let $w_0$ be its cyclic reduction. Then $\widehat{\rho_v}(w)$ counts the cyclic occurrences of $v$ in the reduced cyclic word obtained by closing $w_0$. With this definition understood we have:
\begin{lem}\label{HomogenizationGroup} The counting function $\rho_v: F_n \to \Z$ is homogenizable, and its homogenization is given by the cyclic counting function $\widehat{\rho_v}: F_n \to \Z$.
\end{lem}
\begin{proof} We observe first that if $w = xw_0x^{-1}$ as above, then
\[
\lim_{n \to \infty} \frac{\rho_v(w^n)}{n} = \lim_{n \to \infty} \frac{\rho_v(xw_0^nx^{-1})}{n} =  \lim_{n \to \infty} \frac{\rho_v(w_0^n)}{n},
\]
because $|\rho_v(xw_0^nx^{-1}) - \rho_v(w_0^n)|< 2 |x|_S$. It thus suffices to show that the homogenization of $\rho_v$ coincides with $\widehat{\rho_v}$ on cyclically reduced words $w=w_0$. However, on such words we can argue literally as in the monoid case.
\end{proof}
\begin{example}\label{ExampleHomogenization}
We have $\rho_{a_1}(a_1^na_2a_1^{-n}) = n$, whereas $\widehat{\rho_{a_1}}(a_1^na_2a_1^{-n}) = 0$. Thus $\|\rho_{a_1} - \widehat{\rho_{a_1}}\|_\infty = \infty$.
\end{example}
In the body of the text we will apply homogenization in the following form.
\begin{cor}\label{HomogenizationConvenient}
Let $f = \sum a_w \rho_w$ be a finite sum of counting functions either on the free monoid $M_n$ or on the free group $F_n$. Then the following hold:
\begin{enumerate}[(i)]
\item $f$ is homogenizable.
\item The homogenization of $f$ is given in terms of cyclic counting functions as
\[
\widehat{f} = \sum \alpha_w \widehat{\rho_w}.
\]
\item If $f$ is a bounded function, then $\widehat{f} = 0$.
\end{enumerate}
\end{cor}
\begin{proof} (i) follows from Lemma \ref{HomogenizationMonoid} and Lemma \ref{HomogenizationGroup} together with the fact that homogenizable functions form a vector space. (ii) follows from these lemmas together with the fact that homogenization is linear. (iii) is immediate from the definition of homogenization.
\end{proof}
Finally, let us relate the spaces $\widehat{\mathcal C}(F_n)$ and $\widehat{\mathcal C}(M_n)$ to cyclic counting functions. By Corollary \ref{HomogenizationConvenient} we have well-defined linear maps
\[
\iota_{M_n}: \widehat{\mathcal C}(M_n) \to {\rm span}\{\widehat{\rho_w}\mid w \in M_n\}, \quad \iota_{F_n}: \widehat{\mathcal C}(F_n) \to {\rm span}\{\widehat{\rho_w}\mid w \in F_n\},
\]
which send a class $[f]$ to the homogenization $\widehat{f}$.
\begin{thm}\label{ThmCyclicFcts}
The maps $\iota_{M_n}$ and $\iota_{F_n}$ are isomorphisms. In particular $\widehat{\mathcal C}(F_n)$ and $\widehat{\mathcal C}(M_n)$ can be identified with the respective vector spaces spanned by cyclic counting function.
\end{thm}
Some parts of Theorem \ref{ThmCyclicFcts} are obvious. Firstly, surjectivity is immediate from Corollary \ref{HomogenizationConvenient}. In the monoid case, injectivity is also easy: Namely, given $f \in \mathcal C(M_n)$ we have $\|f-\widehat{f}\|_\infty < \infty$. Thus if $\widehat{f} = 0$, then $f$ is bounded and thus $[f] = 0$. This shows that $\ker(\iota_{M_n})$ is trivial, and thus $\iota_{M_n}$ is indeed an isomorphism. However, in view of Example \ref{ExampleHomogenization} this simple argument does not work in the group case. Instead we have to use the full strength of the proof of Theorem \ref{RelationsObvious}.
\begin{proof}[Proof of Theorem \ref{ThmCyclicFcts}] Since $\widehat{\mathcal C}(F_n)$ is the ascending union of the pure subspaces $\widehat{\mathcal C}(F_n)_L$, it suffices to show that for each fixed $L\geq 2$, the map
\[\iota_L: \widehat{\mathcal C}(F_n)_L \hookrightarrow \widehat{\mathcal C}(F_n) \xrightarrow{\iota_{F_n}}  {\rm span}\{\widehat{\rho_w}\mid w \in F_n\}, \quad [f] \mapsto \widehat{f}\]
is injective. Denote by $q_L:  \R[F_n]_L \to  \widehat{\mathcal C}(F_n)_L$ the natural surjection and define $K'_L(F_n)  := \ker(\iota_L \circ q_L)$. Then we have a commuting diagram with exact rows

\[\begin{xy}\xymatrix{
0\ar[r]&K_L(F_n) \ar[r] \ar[d]^{(\iota_L)_*} & \R[F_n]_L \ar[r]^{q_L} \ar[d]^{\rm Id}& \widehat{\mathcal C}(F_n)_L \ar[d]^{\iota_L} \ar[r] & 0\\
0\ar[r]&K'_L(F_n) \ar[r] & \R[F_n]_L \ar[r]^{\iota_L \circ q_L} &{\rm Im}(\iota_L) \ar[r] & 0.
}\end{xy}
\]
Since $\iota_L$ is onto, the induced map $(\iota_L)_*$ embeds $K_L(F_n)$ into $K'_L(F_n)$, and we have to show that this embedding is onto. We have seen in \textsc{Step 2} of Subsection \ref{SubsecRelGroups} that the kernel $K_L(F_n)$ can be characterized as the subset of $\R[F_n]$ on which certain certificates $\langle c \rangle_L$ vanish. However, by definition of these certificates, these also vanish on $K'_L(F_n)$. This yields the desired surjectivity and finishes the proof.
\end{proof}
Combining Theorem \ref{ThmCyclicFcts} with Theorem \ref{ThmBasis} we deduce:
\begin{cor}[Basis theorem for cyclic counting functions]
\begin{enumerate}[(i)]
\item Denote by $W$ the set of all words in $M_n$ which do not start or end with $a_1$ (including the empty word). Then the cyclic counting functions $\{\widehat{\rho_w}\mid w \in W\}$ form a basis for the space ${\rm span}\{\widehat{\rho_w}\mid w \in M_n\}$.
\item Denote by $W'$ the set of all reduced words in $F_n$ which do not start with $a_1$ or $a_2a_1^{-1}$ and do not end with $a_1^{-1}$ or $a_1a_2^{-1}$ (including the empty word),  and let $W := W' \cup \{a_1^{-1}\}$.  Then the cyclic counting functions $\{\widehat{\rho_w}\mid w \in W\}$ form a basis for the space ${\rm span}\{\widehat{\rho_w}\mid w \in F_n\}$.\end{enumerate}\qed
\end{cor}

\bigskip

\noindent\textbf{Authors' addresses:}\\
\noindent \textsc{Mathematics Department, Technion, Haifa 32000, Israel}\\
\texttt{hartnick@tx.technion.ac.il};

\noindent \textsc{Steklov Mathematical Institute of Russian Academy of Sciences, \\ Gubkina Str. 8, 119991, Moscow, Russia}\\
\texttt{altal@mi.ras.ru}.

\end{document}